\documentclass[a4paper,10pt]{article}
\usepackage{amsmath, amssymb, latexsym, amsfonts, amsthm}
\usepackage{shuffle}
\usepackage{fullpage}
\usepackage{graphicx}
\usepackage{url}
\usepackage{dsfont}
\usepackage{xypic}
\newtheorem{theorem}{Theorem}[section]

\newtheorem{rlemma}[theorem]{Reduction Lemma}
\newtheorem{lemma}[theorem]{Lemma}
\newtheorem{corollary}[theorem]{Corollary}
\newtheorem{problem}[theorem]{Problem}
\newtheorem{proposition}[theorem]{Proposition}

\DeclareMathSymbol\sh{\mathbin}{Shuffle}{"001}
\numberwithin{equation}{section}
\title{\bf On twin and anti-twin words in the support of the free Lie algebra \\}
\author{\bf Ioannis C. Michos \thanks{E-mail address: imichos@uoi.gr}\\ \\
\small{Department of Mathematics, University of Ioannina}, \\
\small{GR-$451\,10$, Ioannina, Greece}}
\date{}
\begin{document}
\maketitle
%\footnotesize
\begin{abstract}

Let ${\mathcal L}_{K}(A)$ be the free Lie algebra on a finite alphabet $A$ over a commutative ring $K$ with unity. For a word $u$ in the free monoid $A^{*}$ let $\tilde{u}$ denote its reversal.
Two words in $A^{*}$ are called twin (resp. anti-twin) if they appear with equal (resp. opposite) coefficients in each Lie polynomial.
Let $l$ denote the left-normed Lie bracketing and $\lambda$ be its adjoint map with respect to the canonical scalar product on the
free associative algebra $K \langle A \rangle$.
Studying the kernel of $\lambda$ and using several techniques from combinatorics on words and the shuffle algebra
$\displaystyle (K{\langle A \rangle},  + , \sh )$,
we show that when $K$ is of characteristic zero two words $u$ and $v$ of common length $n$ that lie in the support of ${\mathcal L}_{K}(A)$ - i.e.,
they are neither powers $a^{n}$ of letters $a \in A$ with exponent $n > 1$ nor palindromes of even length - are twin (resp. anti-twin)
if and only if $u = v$ or $u = \tilde{v}$ and $n$ is odd (resp. $u = \tilde{v}$ and $n$ is even).

\medskip

\end{abstract}

\section{Introduction}

Let $A$ be a finite alphabet, $A^{*}$ be the {\em free monoid} on $A$, $K$ be a commutative ring with unity and $K {\langle A \rangle}$ be the
{\em free associative algebra} on $A$ over $K$. The elements of $K {\langle A \rangle}$ are polynomials on non-commuting variables from $A$ and
coefficients from $K$.
Given two polynomials $P, Q \in K {\langle A \rangle}$, their Lie bracket is defined
as $[P,Q] = PQ - QP$. In this way $K {\langle A \rangle}$ is given a Lie structure. It can be proved (see e.g., \cite[Theorem 0.5]{Reut})
that the {\em free Lie algebra} ${\mathcal L}_{K}(A)$ on $A$ over $K$ is equal to the Lie subalgebra of $K {\langle A \rangle}$ generated by $A$.
Its elements will be called {\em Lie polynomials}. When $K$ is the ring $\mathds Z$ of rational integers, ${\mathcal L}_{K}(A)$ is also known as
the {\em free Lie ring}.
The {\em support} of ${\mathcal L}_{K}(A)$ is the subset of ${A}^{*}$ consisting of those
words that appear (with a nonzero coefficient) in some Lie polynomial.
A pair of words $(u,v)$ is called {\em twin} (resp. {\em anti-twin}) if both words
appear with equal (resp. opposite) coefficients in each Lie polynomial over $K$.

Let $m$ be a positive integer with $m > 1$ and ${\mathds Z}_{m}$ be the ring ${\mathds Z}/(m)$ of integers $\bmod \, m$.
M.-P. Sch\"utzenberger had posed the following problems (see \cite[\S 1.6.1]{Reut}; the last two were pointed to us in a private communication
with G. Duchamp):

\begin{problem} \label{probl:1}
Determine the support of the free Lie ring ${\mathcal L}_{\mathds Z}(A)$.
\end{problem}

\begin{problem} \label{probl:2}
Determine the support of ${\mathcal L}_{{\mathds Z}_{m}}(A)$.
\end{problem}

\begin{problem} \label{probl:3}
Determine all the twin and anti-twin pairs of words with respect to ${\mathcal L}_{{\mathds Z}}(A)$.
\end{problem}

\begin{problem} \label{probl:4}
Determine all the twin and anti-twin pairs of words with respect to ${\mathcal L}_{{\mathds Z}_{m}}(A)$.
\end{problem}
%In view of these problems Sch\"utzenberger considered, for each word $w \in {A}^{*}$, the smallest non-negative integer - which we denote by
%$c(w)$ - that appears as a coefficient of $w$ in some Lie polynomial over $\mathds Z$.

G. Duchamp and J.-Y. Thibon gave a complete answer to Problem~\ref{probl:1} in \cite{Duch + Thib} and proved that the complement of the support
of ${\mathcal L}_{{\mathds Z}}(A)$ in $A^{*}$ consists either of powers $a^{n}$ of a letter $a$, with exponent $n > 1$,
or {\em palindromes} (i.e., words $u$ equal to their {\em reversal} $\tilde{u}$) of even length.
This result was extended in \cite{Duch + Laug + Luqu} - under certain conditions - to {\em traces},
i.e., partially commutative words
%(see \cite{Duch + Krob} for an exposition of trace theory)
and the corresponding free partially commutative Lie algebra
(also known as {\em graph Lie algebra}).

\medskip

In this article we continue the work carried out in \cite{Mich} and answer Problem~\ref{probl:3} by the following result.

\begin{theorem} \label{theo:1}
Two words $u$ and $v$ of common length $n$ that lie in the support of the free Lie algebra over a ring $K$ of characteristic zero, 
i.e., which are neither powers $a^{n}$ of letters $a \in A$ with exponent $n > 1$ nor palindromes of even length, are twin (resp. anti-twin) 
if and only if $u = v$ or $u = \tilde{v}$ and $n$ is odd (resp. $u = \tilde{v}$ and $n$ is even).
\end{theorem}

This had been already conjectured in \cite{Mich}. There we had related Problems~\ref{probl:1} up to ~\ref{probl:4} with the notion of the adjoint endomorphism $l^{*}$ - denoted by $\lambda$ here, for brevity - of the left-normed (left to right) Lie bracketing $l$ in ${\mathcal L}_{K}(A)$ with respect to the canonical scalar product on $K \langle A \rangle$.
Our starting point was the simple idea that a word $w$ does not lie in the support of ${\mathcal L}_{K}(A)$ if and only if ${\lambda}(w) = 0$ and
a pair $(u, v)$ of words is twin (resp. anti-twin) if and only if ${\lambda}(u) = {\lambda}(v)$ (respectively ${\lambda}(u) = - \, {\lambda}(v)$).
%We had also showed that $c(w)$ is either zero or the greatest common divisor of the coefficients of the
%monomials appearing in ${\lambda}(w)$, for the left-normed Lie bracketing $l$ of the free Lie ring.
In view of these, to prove Theorem~\ref{theo:1} it is enough to prove the following result, which had been stated as a conjecture
\cite[Conjecture 2.9]{Mich}.

\begin{theorem} \label{theo:2}
Let $\lambda$ be the adjoint endomorphism of the left-normed Lie bracketing of the free Lie algebra over a ring $K$ of characteristic zero 
with respect to the canonical scalar product on $K \langle A \rangle$ and let $u, v$ be words of common length $n$ such that both 
${\lambda}(u)$ and ${\lambda}(v)$ are non-zero. Then
\begin{description}
 \item[\em{(i)}]
${\lambda}(u) = {\lambda}(v)$ if and only if $u = v$ or $n$ is odd and $u = \tilde{v}$.
 \item[\em{(ii)}]
${\lambda}(u) = - {\lambda}(v)$ if and only if $n$ is even and $u = \tilde{v}$.
\end{description}
\end{theorem}

As we had pointed out in \cite[\S 4]{Mich} Problems~\ref{probl:1} up to \ref{probl:4} may be restated as particular combinatorial questions
on the group ring $K {\mathfrak S}_{n}$ of the symmetric group ${\mathfrak S}_{n}$ on $n$ letters.
The main idea is to view a word $w$ of length $n$ on a fixed sub-alphabet $B = \{ a_{1}, a_{2}, \ldots , a_{r} \}$ of $A$ as an
ordered set partition of $[n] = \{ 1, 2, \ldots , n \}$ denoted by $\{ w \} = (I_{1}(w), I_{2}(w), \ldots , I_{r}(w))$,
where for each $k$ the set $I_{k}(w)$ consists of the positions of $[n]$ in which the letter $a_k$ occurs in $w$.
If $\mu = ({\mu}_{1}, {\mu}_{2}, \ldots , {\mu}_{r})$ is the multi-degree of $w$ then $\{ w \}$ is just a $\mu$-tabloid, where $\mu$ may be,
without loss of generality, assumed an integer partition of $n$.
The role of the reversal $\tilde{w}$ of $w$ is then played by the tabloid ${{\tau}_{n}} \cdot \{ w \}$, where ${\tau}_{n}$ is the involution
$\displaystyle \prod_{i=1}^{k} \, (i, \, n-i+1)$ of ${\mathfrak S}_{n}$ with $k = {\left\lfloor n/2 \right\rfloor}$.
Viewing each permutation as a word in $n$ distinct letters, the left-normed  multi-linear Lie bracketing
$l_{n} = l(x_1 x_2 \cdots x_n)$ and its adjoint ${\lambda}_{n} = {\lambda}(x_1 x_2 \cdots x_n)$
can be viewed as elements of the group ring $K {\mathfrak S}_{n}$; the first one is known as the {\em Dynkin operator}.
The right permutation action of ${\lambda}_{n}$ on words is then equivalent to the left natural action of $l_n$
on tabloids; in particular $w \cdot {\lambda}_{n} \: = \: 0$ if and only if $l_{n} \cdot \{w\} \: = \: 0$.
In this way all results and problems on words translate to the corresponding ones on tabloids.
Theorem~\ref{theo:2}, in particular, has the following equivalent form.

\begin{theorem}\label{dynktabloid}
Let $\mu$ be a partition of $n$, $l_n$ be the Dynkin operator of the free Lie algebra over a ring $K$ of characteristic zero and 
$t_{1}, \, t_{2}$ be $\mu$-tabloids with both $l_{n} \, \cdot \, t_1$ and $l_{n} \, \cdot \, t_2$ different from zero. 
If ${\tau}_{n}$ denotes the involution $\displaystyle \prod_{i=1}^{{\left\lfloor \frac{n}{2} \right\rfloor}} \, (i, \, n-i+1)$ 
of the symmetric group ${\mathfrak S}_{n}$ then
\begin{description}
 \item[\em{(i)}]
$l_{n} \, \cdot \, t_1 \: = \: l_{n} \, \cdot \, t_2$ \, if and only if \, $t_1 = t_2$ or $n$ is odd and \, $t_1 = {{\tau}_{n}} \, \cdot \, t_2$.
 \item[\em{(ii)}]
$l_{n} \, \cdot \, t_1 \: = \: - \,\, l_{n} \, \cdot \, t_2$ \, if and only if $n$ is even and \, $t_1 = {{\tau}_{n}} \, \cdot \, t_2$.
\end{description}
\end{theorem}

\smallskip

Our main objective is to prove Theorem~\ref{theo:2}.
Studying the way $\lambda$ is affected by literal morphisms from an arbitrary finite alphabet to a two lettered one, we had also shown
\cite[Reduction Theorem 2.10]{Mich} that it suffices to work over an alphabet of two letters, so, without loss of
generality, we may assume that $A = \{ a, b \}$.

The paper is organized as follows.

In \S 2 we set up notation and the necessary background on combinatorics on words and free Lie algebras.
We then recall the main ideas of the work in \cite{Mich}. The most important is the study of $\ker \lambda$ which immediately
connects Problems~\ref{probl:1} up to \ref{probl:4} with the shuffle algebra. More precisely, $\ker \lambda$ is equal to the orthogonal
complement ${{\mathcal L}_{K}(A)}^{\perp}$, which, when $K$ is a commutative ${\mathds Q}$-algebra, is also equal, by a result originally due to
R. Ree \cite{Ree}, to the $K$-span of proper shuffles (shuffles where both words are non empty). 
The extra tool which had not been considered in \cite{Mich} and plays a crucial role here is the fact that the map that sends 
an arbitrary polynomial to its {\em right} or {\em left residual} by a Lie polynomial is a derivation in the shuffle algebra.

In \S 3 we present all the necessary results for calculating ${\lambda}(w)$. This is done recursively in two ways.
One may be viewed as a generalization, in terms of polynomials, of the Pascal triangle and leads to the calculation of two certain 
arithmetic invariants of the polynomial ${\lambda}(w)$, namely the non-negative integers $e(w)$ and $d(w)$.
The other one is in terms of all factors $u$ of $w$ of fixed multi-degree and the shuffle product of words.
The equation ${\lambda}(w) \, = \, \pm \, {\lambda}(w')$, after factoring out a common term of a certain form, finally leads us to a certain 
equation in the shuffle algebra.
Rearranging its terms by collecting all non proper shuffles on the one side and all proper shuffles on the other, we immediately know that
the term consisting of non proper shuffles lies in $\ker \lambda$. This allows us to build some sort of inductive argument for the proof 
of Theorem~\ref{theo:2}. 
If all shuffles are proper we may first have to act on the left or on the right by a suitable Lie polynomial.

The remainder of the paper deals directly with the proof of Theorem~\ref{theo:2} where the above ideas will be put in use.
We argue inductively on the length of the corresponding words. This is done in a case by case analysis.
There are three main cases to study: (1) $w = aub$, $w' = avb$; (2) $w = aub$, $w' = ava$; and (3) $w = aua$, $w' = ava$.
For case (1), which is dealt in \S 4, we need nothing more but the definition of the adjoint map $\lambda$ and the 1st and 2nd Theorem
of Lyndon and Sch\"utzenberger. We deal with case (2) in \S 6. It turns out that Case (3) further breaks up in four particular subcases, which
are presented in \S 5, 7, 8 and 9, respectively. These are glued together in \S 10, which is the final touch of the proof.

\bigskip

%%%%%%%%%%%%%%%%%%%%%%% MAIN IDEAS FOR \lambda %%%%%%%%%%%%%%%%%%%%%%%%%%%%%%%%%%%%%%%%%%%%%%%%%%%%%%%%%%%%%%%%%%%
%%%%%%%%%%%%%%%%%%%%%%%%%%%%%%%%%%%%%%%%%%%%%%%%%%%%%%%%%%%%%%%%%%%%%%%%%%%%%%%%%%%%%%%%%%%%%%%%%%%%%%%%%%%%%%%

\section{Preliminaries}

We start with some preliminaries on combinatorics on words (the standard reference is \cite[\S 1]{Loth}).
Let $A = \{ a_{1}, a_{2}, \ldots , a_{q} \}$ be a finite
alphabet totally ordered by $a_{1} < a_{2} < \cdots < a_{q}$ and $A^{*}$ denote the free monoid on $A$.
An element $w$ of $A^{*}$ is a word $w = x_{1}x_{2} \, \cdots \, x_{n}$, i.e.,
a finite sequence of letters from $A$. Its {\em length}, denoted by $|w|$, is the number $n$ of the letters in it. The {\em empty word},
denoted by $\epsilon$, is the empty sequence with length $|{\epsilon}| = 0$. The set of all non-empty words over $A$ is
denoted by $A^{+}$. If $a \in A$ and $w \in A^{*}$ the number of occurrences of the letter $a$ in $w$ is denoted by
${|w|}_{a}$. The set of distinct letters that occur in a word $w$ is denoted by $alph(w)$. The {\em multi-degree} ${\alpha}(w)$ of $w$
may be defined as the vector ${\alpha}(w) = ({|w|}_{a_{1}}, {|w|}_{a_{2}}, \ldots , {|w|}_{a_{q}})$.
The binary operation in $A^{*}$ is the {\em concatenation product}; if $x, y \in A^{*}$ the product $xy$
is the word which is constructed by concatenating $x$ with $y$.
A {\em factor} of a word $w$ is a word $u$ such that $w = sut$ for some $s, t \in A^{*}$. It is called a
{\em right factor} if $t = \epsilon$ and is called {\em proper} if $u \neq w$. Left factors of
a word $w$ are defined in a similar manner. A word $w \in A^{+}$ is called {\em primitive} if it can not be written in the form $w = u^{n}$ for
some $u \in A^{+}$ and $n > 1$. Two words $x,y$ are {\em conjugate} if there exist words $u, v$ such that $x = uv$ and $y = vu$,
i.e., they are cyclic shifts of one another.
The {\em reversal} $\widetilde{w}$ of a given word $w = x_{1}x_{2} \, \cdots \, x_{n}$ is defined as
$\widetilde{w} = x_{n} \, \cdots \, x_{2}x_{1}$. A word is called {\em palindrome} if $\widetilde{w} = w$.

Our main tools from combinatorics on words will be the following two fundamental results due to Lyndon and Sch\"utzenberger
(see \cite[\S 1.3]{Loth}).

\begin{lemma}[1st Theorem of Lyndon and Sch\"utzenberger] \label{1thL+S}
Let $x, y \in A^{+}$ and $z \in A^{*}$. Then $xz = zy$ if and only if there exist $u, v \in A^{*}$ and an integer
$d \geq 0$ such that $x = uv$, $y = vu$ and $z = u{(vu)}^{d}$.
\end{lemma}

\begin{lemma}[2nd Theorem of Lyndon and Sch\"utzenberger] \label{2thL+S}
Let $x, y \in A^{+}$. Then $xy = yx$ if and only if there exist $z \in A^{+}$ and integers $m, n > 0$ such that
$x = z^{m}$ and $y = z^{n}$.
\end{lemma}

The {\em lexicographic order} is a total order on $A^{+}$, also denoted by $<$, which is defined in the following way:
For $u, v \in A^{+}$ we say that $u < v$ if either $v = uw$ for some $w \in A^{+}$
or if there exist $x,y,z \in A^{*}$ and $a, b \in A$ with $a < b$ such that $u = xay$ and $v = xbz$.

A word $w \in A^{+}$ is called {\em Lyndon} (see \cite[\S 5.1]{Loth}) if it is strictly smaller than any of its proper right factors
with respect to the lexicographic order.
We denote the set of all Lyndon words on $A$ by ${\mathcal Lynd} = {\mathcal Lynd}(A)$.
The {\em standard factorization} of a word $w \in {\mathcal Lynd} \setminus A$ is the unique pair $(l, m)$ such
that $w = lm$, where $l,m \in {\mathcal Lynd}$, $l < m$ and $m$ is of maximal length.
%Using the standard factorization
%of Lyndon words one arrives at the unique decomposition of a given word $w \in A^{+}$ as a product of powers of Lyndon
%words in descending order, i.e.,
%\begin{equation}
%w \: = \: l_{1}^{i_{1}} l_{2}^{i_{2}} \cdots l_{k}^{i_{k}},
%\end{equation}
%where $l_{1} \succ l_{2} \succ \cdots \succ l_{k}$ and $i_{1}, i_{2}, \ldots , i_{k} \in {\mathds N}$
%(see \cite[p. 67]{Loth}).

\medskip

Each polynomial $P \in K {\langle A \rangle}$ is written in the form
$P \, = \, \sum_{w \in A^{*}} \, (P, w) \, w$, where $(P, w)$ denotes the coefficient of the word $w$ in $P$.
Its {\em support} $supp(P)$ is the set of all $w \in A^{*}$ with $(P, w) \neq 0$.
$P$ is called {\em homogeneous of degree} $n$ (resp. {\em multi-homogeneous of multi-degree}
$({\alpha}_{1}, {\alpha}_{2}, \ldots , {\alpha}_{q})$) if each $w \in supp(P)$ is of length $n$
(resp. of multi-degree $({\alpha}_{1}, {\alpha}_{2}, \ldots , {\alpha}_{q})$).
The palindrome $\widetilde{P}$ of $P$ is defined as $\widetilde{P} \, = \, \sum_{w \in A^{*}} \, (P, w) \, \tilde{w}$.

The {\em support} of ${\mathcal L}_{K}(A)$ is the subset ${\mathcal S}_{K}(A)$ of ${A}^{*}$ defined as
\begin{equation}
{\mathcal S}_{K}(A) \: = \: \{ w \in A^{*} \; : \; \exists P \, \in {\mathcal L}_{K}(A) \; \mbox{with} \; (P, w) \neq 0 \}.
\end{equation}
Its complement in $A^{*}$, i.e., the set of all words which do not appear in Lie polynomials over $K$, will be
denoted by ${\mathcal V}_{K}(A)$.
Two words $u$ and $v$ are called {\em twin} (respectively {\em anti-twin}) if $(P, u) = (P, v)$
(resp. $(P, u) = -(P, v)$) for each $P \in {\mathcal L}_{K}(A)$.

The set $K {\langle A \rangle}$ becomes a non-commutative associative algebra with unit equal to ${\mathbf 1} \; = \; 1 \, \epsilon$ under the
usual {\em concatenation product} defined as
\begin{equation}
(PQ, \, w) \: = \: \sum_{w = uv} (P,u)(Q,v).
\end{equation}
Given two polynomials $P, Q \in K {\langle A \rangle}$ there is a {\em canonical scalar product} defined as
\begin{equation}
(P,Q) \: = \: \sum_{w \in {A}^{\, *}} (P,w)(Q,w), \label{scalarproduct}
\end{equation}
in the sense that it is the unique scalar product on $K {\langle A \rangle}$ for which ${A}^{*}$
is an orthonormal basis.

The set $K {\langle A \rangle}$ becomes also a commutative associative algebra with unit ${\mathbf 1}$ under the
{\em shuffle product} that is initially defined for words as
${\epsilon} \, \sh \: w  \: = \: w  \, \sh \, {\epsilon} \: = \: w$, if at least one of them is the empty word,
and recursively as
\begin{equation}
(au') \, \sh \, (bv') \: = \: a(u' \, \sh \, (bv')) \, + \, b((au') \, \sh \, v'), \label{recdefshuffle}
\end{equation}
if $u = au'$ and $v = bv'$ with $a, b \in A$ and $u', v' \in A^{*}$ (see \cite[(1.4.2)]{Reut}).
It is then extended linearly to the whole of $K {\langle A \rangle}$ by the formula
\begin{equation}
P \; \sh \; Q \: = \: \sum_{u, v \in A^{*}} (P,u)(Q,v) \; u \, \sh \, v.
\end{equation}
It is easily checked, initially for words, that
\begin{equation} \label{Preversalshuffle}
\widetilde{PQ} \, = \, \widetilde{Q} \widetilde{P} \quad \text{and} \quad \widetilde{P \sh Q} \, = \, \widetilde{P} \sh \widetilde{Q}.
\end{equation}

The shuffle product $u \, \sh \, v$ of two words $u$ and $v$ is called {\em proper} if $u, v \in A^{+}$.

\medskip

The {\em right and the left residual} of a polynomial $P$ by another polynomial $Q$, denoted respectively by
$P \, \rhd \, Q$ and $Q \, \lhd \, P$, are the polynomials defined by the formulae
\begin{equation}
(P \, \rhd \, Q, \, w) \; = \; (P, \, Qw) \quad \text{and} \quad (Q \, \lhd \, P, \, w) \; = \; (P, \, wQ),
\end{equation}
for each $w \in A^{*}$.
The term right and left residual comes from the easily checked identities $uv \, \rhd \, u \; = \; v$ and
$u \, \lhd \, vu \; = \; v$, for each $u,v \in A^{*}$. For a fixed $Q \in K{\langle A \rangle}$ the maps
$P \mapsto P \, \rhd \, Q$ and $P \mapsto Q \, \lhd \, P$ are easily checked to be $K$-linear.
Furthermore one can also check that $P \, \rhd \, (QR) \; = \; (P \, \rhd \, Q) \, \rhd \, R$ and
$P \, \rhd \, {\mathbf 1} \; = \; P$. This shows that the right residual $\rhd$ is a right action of
the algebra $K {\langle A \rangle}$ on itself. Similarly one can also verify that
$(RQ) \, \lhd \, P \; = \; R \, \lhd \, (Q \, \lhd \, P)$ and ${\mathbf 1}  \, \lhd \, P \; = \; P$, so that
the left residual $\lhd$ is a left action of the algebra $K {\langle A \rangle}$ on itself.
The following result, which we first saw in an explicit form in \cite{Minh + Peti}, is crucial to our work.

\begin{proposition} {\rm (\cite[Lemma 2.1]{Minh + Peti})}
\label{minhpetitot}
The right and the left residual by a Lie polynomial are derivations of the shuffle algebra, i.e., if $P$ and $Q$ are
polynomials in $K {\langle A \rangle}$ and $R$ is a Lie polynomial then we have
\begin{eqnarray*}
 (P \, \sh \, Q) \, \rhd \, R \: = \: (P \, \rhd \, R) \; \sh \; Q \: + \: P \, \sh \; (Q \, \rhd \, R); \\
 R \, \lhd \, (P \, \sh \, Q) \: = \: (R \, \lhd \, P) \; \sh \; Q \: + \: P \, \sh \; (R \, \lhd \, Q).
\end{eqnarray*}
\end{proposition}

\begin{proof} It suffices to show these identities when $R$ is a homogeneous Lie polynomial of degree $n$. The case
$n = 1$ follows directly from \eqref{recdefshuffle} (cf. \cite[p. 26]{Reut}).
For the induction step we set $R = [R', \, a]$, where $a \in A$ and $R'$ is a homogeneous Lie polynomial of degree $n$.
Then the result follows since the right and the left residual are respectively a right and a left action of
$K {\langle A \rangle}$ on itself.
\end{proof}

We will also need the following two technical results.

\begin{lemma} \label{abshuffle}
Let $m, n \in {\mathds N}$. Then
\[ a^{m} \: \sh \: ba^{n} \: = \: \sum_{i=0}^{m} \binom{n+i}{i} \, a^{m-i}ba^{n+i}. \]
\end{lemma}

\begin{proof}
By an easy induction on $n$.
\end{proof}

\smallskip

\begin{lemma} \label{PdPe}
Let $K$ be a commutative ring with unity and $P$ be a non zero polynomial of multi-degree $(n,m)$ in $K \langle a , \, b \rangle$,
where $n, m \geq 1$. There exist unique non-negative integers $d = d(P)$ and $e = e(P)$ with $d \leq e$ and unique polynomials
$P_{d}, \, \ldots \, , P_{e}$, with $P_{d} \neq 0$ and $P_{e} \neq 0$, such that $P$ is written in the form
\[ P \: = \: P_{d} \, ba^{d} \: + \: \cdots \: + \: P_{e} \, ba^{e} \, . \]
\end{lemma}

\begin{proof}
Let $\displaystyle 0 \neq P = \sum_{u \in A^{+}} k_{u} u \, \in \, K \langle a, b \rangle$. For each $u \in supp(P)$ there exists a unique
non negative integer $i$ such that $u = vba^{i}$. Using the notion of the left residual we get $v = ba^{i} \lhd u$.
Consider the finite set $I = \{ i \geq 0 \, : \, \exists \, u \in supp(P) \, \mbox{with} \, u = vba^{i} \}$.
For a fixed $i \in I$ define $X_{i} = \{ u \in supp(P) \, : \, \exists \, v \in \{a, b \}^{*} \, \mbox{with} \, u = vba^{i} \}$.
Clearly $supp(P)$ is equal to the disjoint union $\bigcup_{i \in I} X_{i}$. Then we obtain
\[ P = \sum_{u \in supp(P)} k_{u} u = \sum_{i \in I} \bigl( \sum_{u \in X_{i}} k_{u} (ba^{i} \lhd u) \bigr) ba^{i}. \]
Our result follows if we choose $\displaystyle P_{i} = \sum_{u \in X_{i}} k_{u} (ba^{i} \lhd u)$, $d = \min{I}$ and $e = \max{I}$.
\end{proof}

If we interchange the role of the letters $a$ and $b$, the polynomial $P$ can also be written uniquely in an analogous form. To distinguish 
the invariants $d(P)$ and $e(P)$ in this case from the default one of Lemma~\ref{PdPe} we note them as $d_{b}(P)$ and $e_{b}(P)$.  
Now suppose that a word $w$ lies in the support of the free Lie algebra, i.e., ${\lambda}(w) \neq 0$. In view of Lemma~\ref{PdPe} we define
$d(w) = d({\lambda}(w))$ and $e(w) = e({\lambda}(w))$.

\bigskip

The {\em left-normed Lie bracketing} of a word is the Lie polynomial defined recursively as
\begin{equation}
{\ell}({\epsilon}) \, = \, 0 \, , \qquad {\ell}(a) \, = \, a \, , \qquad \mbox{and} \qquad
{\ell}(ua) \, = \,  [{\ell}(u) \, , a] \, ,
\end{equation}
for each $a \in A$ and $u \in {A}^{+}$.
One can extend $\ell$ linearly to $K {\langle A \rangle}$ and construct a linear map, denoted also by
$\ell$, which maps
$K {\langle A \rangle}$ onto the free Lie algebra ${\mathcal L}_{K}(A)$, since the set
$\{ {\ell}(u) \: : \: u \in A^{*} \}$ is a well known $K$-linear generating set of ${\mathcal L}_{K}(A)$
(see e.g. \cite[\S 0.4.1]{Reut}).

With any word $w \in {\mathcal Lynd}$, we associate its {\em bracketed form}, denoted by $[w]$, which is a Lie polynomial
defined recursively as follows:
\begin{equation}
[w]  = \begin{cases}
   a,  &\text{if \, $w = a, \quad a \in A$;} \\
    [[l], \, [m]],  &\text{if \, $(l, m)$ is the standard factorization of $w \in {\mathcal Lynd} \setminus A$ .}
               \end{cases}
\end{equation}
It is well known (e.g., see \cite[\S 5.3]{Loth}) that the set $\{ [l] \; : \; l \in {\mathcal Lynd} \}$ is a $K$-basis
of ${\mathcal L}_{K}(A)$ with the triangular property
\begin{equation} \label{1Lynd}
[l] \: = \: l \: + \: \sum_{w \in A^{*}} a_{w} \, w,
\end{equation}
where $a_{w} \in K$, $|w| = |l|$ and $w > l$.
%It follows that if $l, l' \in {\mathcal Lynd}$ and $|l| = |l'|$ then
%\begin{equation}
%l \; \rhd [l'] \: = \: [l'] \; \lhd l \: = \: \begin{cases}
%  1, &\text{if \, $l = l'$;} \\
%  0, &\text{if \, $l \prec l'$.}
%   \end{cases}
%\end{equation}

The adjoint endomorphism ${\ell}^*$ of the left-normed Lie bracketing $\ell$,
which will be denoted by $\lambda$ for brevity, is then defined by the relation
\begin{equation}
({{\lambda}}(u), v) \, = \, ({\ell}(v), u), \label{defl*}
\end{equation}
for any words $u$, $v$. The image of $\lambda$ on a word of ${A}^{\, *}$ can also be effectively defined recursively
by the relations
\begin{equation}
{{\lambda}}({\epsilon}) \, = \, 0 \, , \qquad {{\lambda}}(a) \, = \, a \, , \qquad \mbox{and} \qquad
{{\lambda}}(aub) \: = \: {{\lambda}}(au)\,b - {{\lambda}}(ub)\,a \, , \label{defrecl*}
\end{equation}
where $a, b \in A$ and $u \in {A}^*$ (cf. \cite[Problem 5.3.2]{Loth}). The proof goes by induction on the
length of the given word, just as in the case of the adjoint endomorphism of the right-normed Lie
bracketing (discussed in detail in \cite[pp. 32 - 33]{Reut}).

One can also extend $\lambda$ linearly to the whole of $K {\langle A \rangle}$ and construct a linear
endomorphism of $K {\langle A \rangle}$, denoted also by $\lambda$.

\begin{proposition} {\rm (\cite[Proposition 2.5]{Mich})} \label{lambdaLyndon}
Let $K$ be a commutative ring with unity and $l_{1}, l_{2}, \ldots , l_{r}$ be all the Lyndon words of length $n$ on the alphabet $A$.
Then the set $\{ {\lambda}(l_{1}), {\lambda}(l_{2}), \ldots , {\lambda}(l_{r}) \}$ is a $K$-basis of the image under $\lambda$ of the
$n$-th homogeneous component of $K {\langle A \rangle}$.
\end{proposition}

By a straightforward induction on $|w|$ we also have the following.

\begin{lemma} {\rm (\cite[Lemma 2.1]{Mich})} \label{l*reversal}
Let $\tilde{w}$ denote the reversal of the word $w$. Then
\[ {{\lambda}}(\tilde{w}) \, = \, (-1)^{|w| + 1} {{\lambda}}(w)\, \label{tildel*}. \]
\end{lemma}

What is of crucial importance for Problems~\ref{probl:1} up to~\ref{probl:4} is the kernel $\ker {\lambda}$ of $\lambda$.
Let $P, Q \in K {\langle A \rangle}$. We say that $P \, \equiv \, Q$ if $P - Q \in \ker {\lambda}$. We also define $P \, \sim \, Q$ if
there exist $k_{1}, k_{2} \in K^{*} = K \setminus \{ 0 \}$ such that $k_{1} \, P \, \equiv \, k_{2} \, Q$.
Let ${{\mathcal L}_{K}(A)}^{\perp}$ denote the orthogonal complement of ${\mathcal L}_{K}(A)$ with respect to
the scalar product (\ref{scalarproduct}) in $K {\langle A \rangle}$.
Then for an arbitrary commutative ring $K$ with unity the following hold.

\begin{lemma} {\rm (\cite[Lemmas 2.2 \& 2.3]{Mich})} \label{kerlambda}
\begin{description}
\item[\em{(i)}]
$\ker {\lambda} \; = \; {{\mathcal L}_{K}(A)}^{\perp} \; = \; {\mathcal V}_{K}(A)$.
\item[\em{(ii)}]
Two words $u$ and $v$ are twin (resp. anti-twin) with respect to ${\mathcal L}_{K}(A)$ if and only if
$u \, \equiv \, v$ (resp. $u \, \equiv \, - \, v$).
\end{description}
\end{lemma}

In view of Lemma \ref{kerlambda}\,(i), when $K$ is of characteristic zero the notation $u \, \equiv \, 0$ means that $u$ is either a power
$a^{n}$ of a letter $a \in A$ with $n > 1$ or a palindrome of even length.

\medskip

The following result, originally due to Ree \cite{Ree}, is crucial to our work.

\begin{proposition}{\rm (\cite[Theorem 3.1\,(iv)]{Reut}, \cite[Problem 5.3.4\,]{Loth})} \label{Ree}
Let $\mathds Q$ denote the field of rational numbers and assume that $K$ is a commutative
$\mathds Q$-algebra. Then a polynomial $P \in K {\langle A \rangle}$ with zero constant term lies in
${{\mathcal L}_{K}(A)}^{\perp}$ if and only if it is a $K$-linear combination of proper shuffles.
\end{proposition}

\medskip

In view of Problems~\ref{probl:1} up to~\ref{probl:4} Sch\"utzenberger considered, for a given word $w$ of length $n$, the unique non-negative generator ${\gamma}(w)$ of the ideal $\{ (P,w) \, : \, P \in {\mathcal L}_{\mathds Z}(A) \}$ of $\mathds Z$ (see \cite[\S 1.6.1]{Reut}).
He also posed the problem of determining the set of all words $w$ with ${\gamma}(w) = 1$. By \eqref{1Lynd} it follows that Lyndon words
have this property. We will show later on that there are other families of words with ${\gamma}(w) = 1$ or ${\gamma}(w)$ odd; it turns out that
this will be another important tool for Problem~\ref{probl:3}.
The calculation of ${\gamma}(w)$ for a given word $w$ has been given by the following result.

\begin{proposition} {\rm (\cite[Theorem 2.6]{Mich})} \label{gcd}
Let $w$ be a word in ${A}^{+}$ and $\lambda$ be the adjoint endomorphism of the left-normed Lie
bracketing $l$ of the free Lie ring on $A$. If ${\lambda}(w) = 0$ then ${\gamma}(w) = 0$; otherwise it is equal to
the greatest common divisor of the coefficients that appear in the monomials of ${\lambda}(w)$.
\end{proposition}

\bigskip

%%%%%%%%%%%%%%%%%%%%%%%%%%%%%%%%%%%%%%%%%%%%%% PARIS SHUFFLE %%%%%%%%%%%%%%%%%%%%%%%%%%%%%%%%%%%%%%%%%%%%%%%%%%%%%%%%%%%

\section{Calculation of $\mathbf {{\lambda}}$}

Our starting point is the following result for words of multi-degree $(k, 1)$ with $k \geq 0$.

\begin{lemma} {\rm (\cite[Lemma 2.8]{Mich})} \label{lambda1b}
{\em Let $k$ and $l$ be non-negative integers which are not both equal to zero. Then}
\begin{equation*}
{\lambda}(a^{k}ba^{l}) \: = \: {(-1)}^{k} \binom{k  + l }{k} {\lambda}(ba^{k + l})
                   \: = \: {(-1)}^{k} \binom{k  + l }{k} \, \{ ba^{k + l } - aba^{k + l - 1} \} .
%{\gamma}(a^{k}ba^{l}) & = & \binom{k  + l }{k}.
\end{equation*}
\end{lemma}

Using the shuffle product of words the polynomial ${\lambda}(w)$ can be calculated recursively in terms of all factors $u$ of fixed length
$r \geq 1$ of $w$.

\begin{proposition} {\rm (\cite[Proposition 3.1]{Mich})} \label{reclambdashuffle}
Let $w$ be a word and $r$ be a positive integer with $r \leq |w|$.
Consider the set of all factors $u$ of length $r$ of $w$. Then
\[{\lambda}(w) \: = \: \sum_{{w = sut} \atop {|u|=r}} \, {\lambda}(u) \, (-1)^{|s|} \, \{ \tilde {s} \, \sh \, \,  t \} \, . \]
\end{proposition}

The case where the factors $u$ of the word $w$ are letters (i.e., we are at the bottom level $r=1$) seems to be known
in a different setting in the literature (see \cite[Lemma 2.1]{Kats + Koba}, cf. \cite[Ex. 4.6.5\,(2), p.126]{Duch + Krob})
even for the broader class of free partially commutative Lie algebras.

Decomposing further the factors $u$ up to multi-degree we obtain
\begin{equation} \label{multilambdashuffle}
{\lambda}(w) \: = \: \sum_{{w = sut} \atop {|u|=r}} \,
\sum_{{{\alpha}(u) = (k_{1}, k_{2}, \ldots , k_{q})} \atop {k_{1} + k_{2} + \dots + k_{q} = r}} \!\!\!\!\!\!\!\!\!\!\!
{\lambda}(u) \, (-1)^{|s|} \, \{ \tilde {s} \, \sh \, t \} \, .
\end{equation}

Dealing now with the equation $w \, \equiv \, \pm \, w'$ or more generally with the equation $w \, \sim \, w'$, we obtain the following result.

\begin{proposition} \label{multieqlambdashuffle}
Suppose that $w$ and $w'$ are two words of given multi-degree $(m_{1}, m_{2}, \ldots , m_{q})$ on an alphabet with $q$ letters such that
${\eta} \, {\lambda}(w) \, = \, {\eta}' \, {\lambda}(w')$, for some ${\eta}, {\eta}' \in {\mathds Z}$.
Then for all factors of a fixed multi-degree $(k_{1}, k_{2}, \ldots , k_{q})$ we obtain
\begin{equation} \label{keyresult}
{\eta} \!\!\!\!\!\!\!\!\!\!\! \sum_{{w = sut} \atop {{\alpha}(u) = (k_{1}, k_{2}, \ldots , k_{q})}} \!\!\!\!\!\!\!\!\!\!\!
{\lambda}(u) \, (-1)^{|s|} \, \{ \tilde{s} \, \sh \, \,  t \}
\: = \: {\eta}' \!\!\!\!\!\!\!\!\!\!\!
\sum_{{w' = s'u't'} \atop {{\alpha}(u') = (k_{1}, k_{2}, \ldots , k_{q})}} \!\!\!\!\!\!\!\!\!\!\!
{\lambda}(u') \, (-1)^{|s'|} \, \{ \widetilde{s'} \, \sh \, \,  t' \} \, .
\end{equation}
\end{proposition}

\begin{proof}
The polynomials ${\lambda}(u)$ and $\tilde{s} \, \sh \, \,  t$ are multi-homogeneous, so in view of \eqref{multilambdashuffle},
a typical monomial that appears in ${\lambda}(w)$ is of the form ${\eta}_{x,y} \, x \cdot y$, where $x$ is a word
of multi-degree $(d_{1}, d_{2}, \ldots , d_{q})$ appearing in ${\lambda}(u)$, $y$ is a word of multi-degree
$(m_{1} - d_{1}, m_{2} - d_{2}, \ldots , m_{q} - d_{q})$ appearing in $\tilde{s} \, \sh \, \,  t$ and ${\eta}_{x,y} \in \mathds{Z}$.
Similarly a typical monomial in ${\lambda}(w')$ is of the form ${\eta}^{'}_{x,y} \, x' \cdot y'$, where $x'$ is a word of multi-degree
$(e_{1}, e_{2}, \ldots , e_{q})$ appearing in ${\lambda}(u')$, $y$ is a word of multi-degree
$(m_{1} - e_{1}, m_{2} - e_{2}, \ldots , m_{q} - e_{q})$ appearing in $\tilde{s'} \, \sh \, \,  t'$ and ${\eta}^{'}_{x,y} \in \mathds{Z}$.
Since $|u| = |u'|$, the only possible way that the terms ${\eta}_{x,y} \, x \cdot y$ and ${\eta}^{'}_{x,y} \, x' \cdot y'$ cancel each other out
in the equality ${\lambda}(w) \, = \, \pm \, {\lambda}(w')$ is that ${\alpha}(x) = {\alpha}(x')$ and ${\alpha}(y) = {\alpha}(y')$. Thus
we must have $d_{i} = e_{i}$, for each $i = 1, 2, \ldots , q$. Then for factors of a fixed multi-degree the result follows.
\end{proof}

Suppose now that our given words $w, w'$ are over a two lettered alphabet $A = \{ a, b \}$. Fix a multi-degree $(k, \, l)$ and consider all
factors $u$ and $u'$ of $w$ and $w'$ respectively with ${\alpha}(u) = {\alpha}(u') = (k, \, l)$. If at least one of $k,l$ is equal to $1$ - i.e.,
one of the letters $a, b$ appears once - or $k = l = 2$ then the corresponding terms ${\lambda}(a^{k}b)$, ${\lambda}(b^{l}a)$ and ${\lambda}(a^{2}b^{2})$ will factor out of \eqref{keyresult}. This is due to Lemma~\ref{lambda1b} and the easily checked equality
${\lambda}(abab) = - 2 {\lambda}(a^{2}b^{2}) = 2 {\lambda}(b^{2}a^{2}) = - {\lambda}(baba)$, respectively.
In this way we obtain an equality in the shuffle algebra that looks like
\begin{equation}
\sum_{i} \, {\eta}_{i} \, (-1)^{|s_{i}|} \, \{ \widetilde{s_{i}} \, \sh \, t_{i} \} \: = \:
\sum_{i} \, {\eta}^{'}_{i} \, (-1)^{|s^{'}_{i}|} \, \{ \widetilde{s^{'}_{i}} \, \sh \, t^{'}_{i} \} \, ,
\end{equation}
for some integer coefficients ${\eta}_{i}$ and ${\eta}^{'}_{i}$. From this point and on we either use Proposition~\ref{Ree} directly or in more
difficult situations where all shuffles are proper we act on the left or on the right accordingly by suitable Lie polynomials using Proposition~\ref{minhpetitot}.

\medskip

The following result is a recursive formula for the calculation of ${\lambda}(w)$ which is strongly related to the Pascal triangle.

\begin{proposition} \label{lambdapascal}
Let $k, l$ be non negative integers. Then
\begin{equation*}
{\lambda}(a^{k}buba^{l})  \quad = \quad
(-1)^{k+1} \sum_{i=0}^{l} \, \binom{k+i}{i} \, {\lambda}(uba^{l-i}) \, ba^{k+i} \quad + \quad
           \sum_{j=0}^{k} \, (-1)^{j} \binom{l+j}{j} \, {\lambda}(a^{k-j}bu) \, ba^{l+j}.
\end{equation*}
\end{proposition}

\begin{proof}
This result is essentially the same as \cite[Proposition 5.4]{Mich}; it had been stated there in a commutative algebra setting using the notion
of Pascal descent polynomials. For completeness, we give a proof.

Without loss of generality we may assume that $k \leq l$. The case where $k = l = 0$ follows trivially as it is equivalent to
${\lambda}(bub) = {\lambda}(bu)b \, - \, {\lambda}(ub)b$.
The case $k = 0 < l$  is written as
\begin{equation} \label{bubal}
{\lambda}(buba^{l}) \: = \: - \, \sum_{i=0}^{l} \, {\lambda}(uba^{l-i}) \, ba^{i}
                             \: + \: {\lambda}(bu)ba^{l} \, ,
\end{equation}
and follows by an immediate induction on $l$.

Suppose that $0 < k \leq l$. Then the induction is a bit more tedious and is done on $k + l = m$.
For $m = 2$, i.e., $k = l = 1$ the result follows trivially. Suppose that it holds for all pairs
$(k', \, l')$ with $1 \leq k' \leq l'$ and $k' + l' \leq m - 1$. In particular, it holds for the pairs
$(k-1, \, l)$ and $(k, \, l-1)$.
We will show that it will also hold for the pair $(k, \, l)$ with $1 \leq k \leq l$ and $k + l = m$.
For brevity let $w = a^{k}buba^{l}$.
\begin{eqnarray*}
{\lambda}(w)  & = &  {\lambda}(a^{k}buba^{l-1}) \, a  \: - \: {\lambda}(a^{k-1}buba^{l}) \, a \\
& = &  (-1)^{k+1} \sum_{i=0}^{l-1} \, \binom{k+i}{i} {\lambda}(uba^{l-1-i}) \, ba^{k+i+1}
        \: + \: \sum_{j=0}^{k} \, (-1)^{j} \binom{l-1+j}{j} {\lambda}(a^{k-j}bu) \, ba^{l-1+j+1}\\
&  & \!\!\!\!\! -(-1)^{k-1+1} \sum_{i=0}^{l} \, \binom{k-1+i}{i} {\lambda}(uba^{l-i}) \, ba^{k-1+i+1}
        \: - \: \sum_{j=0}^{k-1} \, (-1)^{j} \binom{l+j}{j} {\lambda}(a^{k-1-j}bu) \, ba^{l+j+1} \\
& = &  (-1)^{k+1} \Big\{ \sum_{i=0}^{l-1} \, \binom{k+i}{i} {\lambda}(uba^{l-1-i}) \, ba^{k+i+1}
\, + \, \sum_{i=0}^{l} \, \binom{k-1+i}{i} {\lambda}(uba^{l-i}) \, ba^{k+i} \Big\} \\
&  & \!\!\!\!\! \, + \, \sum_{j=0}^{k} \, (-1)^{j} \binom{l-1+j}{j} {\lambda}(a^{k-j}bu) \, ba^{l+j}
\: + \: \sum_{j=0}^{k-1} \, (-1)^{j+1} \binom{l+j}{j} {\lambda}(a^{k-1-j}bu) \, ba^{l+j+1} \\
& = &  (-1)^{k+1} \Big\{ \sum_{i=1}^{l} \, \Big[ \binom{k+i-1}{i-1} + \binom{k-1+i}{i} \Big] \,
     {\lambda}(uba^{l-i}) \, ba^{k+i} \: + \: \binom{k-1+0}{0} {\lambda}(uba^{l-0}) \, ba^{k+0} \, \Big\} \\
&  & \!\!\!\!\! \: + \: \Big\{ \sum_{j=1}^{k} \, (-1)^{j} \Big[ \binom{l-1+j}{j} + \binom{l+j-1}{j-1} \Big]
  {\lambda}(a^{k-j}bu) \, ba^{l+j} \: + \: (-1)^{0} \binom{l-1+0}{0} {\lambda}(a^{k-0}bu) \, ba^{l+0} \, \Big\} \\
& = &  (-1)^{k+1} \big\{ \sum_{i=1}^{l} \, \binom{k+i}{i} {\lambda}(uba^{l-i}) \, ba^{k+i} \: + \:
   \binom{k + 0}{0}  {\lambda}(uba^{l-0}) \, ba^{k+0} \Big\} \\
&  & \!\!\!\!\! \: + \: \big\{ \sum_{j=1}^{k} \, (-1)^{j} \binom{l+j}{j} {\lambda}(a^{k-j}bu) \, ba^{l+j}
\: + \: (-1)^{0}\binom{l+0}{0} {\lambda}(a^{k-0}bu) \, ba^{l+0} \Big\} \\
& = &  (-1)^{k+1} \sum_{i=0}^{l} \, \binom{k+i}{i} \, {\lambda}(uba^{l-i}) \, ba^{k+i} \quad + \quad
           \sum_{j=0}^{k} \, (-1)^{j} \binom{l+j}{j} \, {\lambda}(a^{k-j}bu) \, ba^{l+j},
\end{eqnarray*}
as required.
\end{proof}

\begin{corollary} \label{Pi}
Suppose that $k \leq l$. Then
\[ {\lambda}(a^{k}buba^{l}) = P_{k+l} \, ba^{k+l} \: + \: P_{k+l-1} \, ba^{k+l-1} \: + \: \cdots \: + \:
                           P_{k+1} \, ba^{k+1} \: + \: P_{k} \, ba^{k} \, , \qquad  \mbox{where} \]
%\begin{math}
\begin{eqnarray}%{lcl}
 P_{k+l}  & = &
\begin{cases}
    (-1)^{k} \displaystyle \binom{k+l}{k} \, \{ {\lambda}(bu) \, - \, {\lambda}(ub) \},
                               & \text{if \, $bub \, \not \equiv \, 0$}  \\
   \displaystyle 0,  &\text{if \, $bub \, \equiv \, 0$;}
               \end{cases}  \label{k+l} \\
P_{k+l-1}  & = &  \begin{cases}
   \displaystyle (-1)^{k-1}\,\Big \{ \binom{k+l-1}{l-1}{\lambda}(uba) \, + \, \binom{k+l-1}{l}{\lambda}(abu) \Big\},
                               &\text{if \, $bub \, \not \equiv \, 0$} \\
   \displaystyle (-1)^{k-1} \binom{k+l}{k} \, {\lambda}(abu),  &\text{if \, $u \, = \, b^{m}, \, m$ odd,} \\
   \displaystyle (-1)^{k-1} \Big\{ \binom{k+l-1}{l} \, - \, \binom{k+l-1}{l-1} \Big\} \, {\lambda}(abu),
                               &\text{if \, $bub \, = \, s{\tilde{s}}, \: s \in A^{*}$;}
               \end{cases}  \label{k+l-1} 
\end{eqnarray}               
\begin{eqnarray}
P_{k+1}  & = &  \begin{cases}
   \displaystyle (-1)^{k+1} (k+1) \, {\lambda}(uba^{l-1}), \, &\text{if \, $k < l$} \\
   \displaystyle (k+1) \, \{ (-1)^{k+1} {\lambda}(uba^{k-1}) \: - \: {\lambda}(a^{k-1}bu) \},  &\text{if \, $k = l$;}
               \end{cases}  \label{k+1} \\
P_{k}  & = &  \begin{cases}
   \displaystyle (-1)^{k+1} {\lambda}(uba^{l}),  &\text{if \, $k < l$} \\
   \displaystyle (-1)^{k+1} {\lambda}(uba^{k}) \: + \: {\lambda}(a^{k}bu),  &\text{if \, $k = l$.}
               \end{cases} \label{k}
\end{eqnarray}
%\end{math}
\end{corollary}

\begin{proof}
It follows as an immediate consequence of Proposition~\ref{lambdapascal}.
\end{proof}

\begin{corollary} \label{ewdw}
Let $k$ and $l$ be non negative integers with $k \leq l$ and $k, l$ not both equal to zero. If $a^{k}buba^{l}$ is a word
which is not a palindrome of even length then

\begin{math}
\begin{array}{lcl}
 e(a^{k}buba^{l})  & = & \begin{cases}
   \displaystyle k + l, \, &\text{if \, $bub \, \not \equiv \, 0$} \\
   \displaystyle k + l - 1, \,  &\text{if \, $bub \, \equiv \, 0$;}
               \end{cases} \\ \\
 d(a^{k}buba^{l})  & = & \begin{cases}
   \displaystyle k, \, &\text{if \, $k < l$ \, and \, $uba^{l} \, \not \equiv \, 0$} \\
   \displaystyle k + 1, \,  &\text{if \, $k < l$ \, and \, $uba^{l} \, \equiv \, 0$} \\
   \displaystyle k, \, &\text{if \, $k = l$ \, and \, ${(-1)}^{k+1}uba^{k} + a^{k}bu \, \not \equiv \, 0$} \\
   \displaystyle \geq k + 1, \, &\text{if \, $k = l$ \, and \, ${(-1)}^{k+1}uba^{k} + a^{k}bu \, \equiv \, 0$.}
               \end{cases}
\end{array}
\end{math}
\end{corollary}

\begin{proof}
It follows directly from Corollary~\ref{Pi}.
\end{proof}

\medskip

\begin{lemma} \label{gamma1}
Let $k \geq 1$ and $m \geq 0$. Then ${\gamma}(b{(a^{2k}b^{2})}^{m}) = 1$.
\end{lemma}

\begin{proof}
We argue by induction on $m$. For $m = 0$ it is trivial to check that ${\gamma}(b) = 1$. For the induction step
we get ${\lambda}(b{(a^{2k}b^{2})}^{m+1}) = - {\lambda}({(a^{2k}b^{2})}^{m+1}) b$, since $b{(a^{2k}b^{2})}^{m}a^{2k}b \equiv 0$.
Then ${\lambda}({(a^{2k}b^{2})}^{m+1}) = {\lambda}({(a^{2k}b^{2})}^{m}a^{2k}b)b \, - \, {\lambda}({a^{2k-1}b^{2}(a^{2k}b^{2})}^{m})a$.
It is enough to show that ${\gamma}({(a^{2k}b^{2})}^{m}a^{2k}b) = 1$. By Lemma~\ref{l*reversal} and \eqref{bubal} we obtain
\begin{equation}
{\lambda} \bigl( b{(a^{2k}b^{2})}^{m}a^{2k} \bigr) \, = \, {\lambda} \bigl( b{(a^{2k}b^{2})}^{m} \bigr) a^{2k} \, - \,
\sum_{i=0}^{2k-1} {\lambda} \bigl( a^{2k}{(b^{2}a^{2k})}^{m-1}b^{2}a^{2k-i} \bigr) ba^{i}.
\end{equation}
By our induction hypothesis ${\gamma}({(b^{2}a^{2k})}^{m}b) = 1$, therefore it clearly follows that
${\gamma} \bigl( b{(a^{2k}b^{2})}^{m}a^{2k} \bigr) = 1$, as required.
\end{proof}

\begin{lemma} \label{gammaodd}
Let $m \geq 0$ and $k \geq 2$. Then ${\gamma}({(ba^{2}ba^{2k})}^{m}b)$ is odd.
\end{lemma}

\begin{proof}
For $m = 0$ this holds as ${\gamma}(b) = 1$.
Let $m = 1$. Collecting all factors of $ba^{2}ba^{2k}b$ of multi-degree $(2k, \, 2)$ we observe that the only term that appears in
\eqref{multilambdashuffle} is ${\lambda}(ba^{2}ba^{2k-2}) \cdot a^{2}b$, so it suffices to show that ${\gamma}(ba^{2}ba^{2k-2})$ is odd.
If $k > 2$ then $a^{2k-2}ba^{2}b \in {\mathcal Lynd}$, hence ${\gamma}(ba^{2}ba^{2k-2}) = 1$.
On the other hand, if $k = 2$ we obtain
${\lambda}(ba^{2}ba^{2}) \; = \; - \, {\lambda}(a^{2}ba) \, \{ b \sh a \} \, + \, {\lambda}(aba^{2}) \, ab \;
= \; - \, {\lambda}(a^{2}ba) \, (2ab + ba) \; = \; -3 {\lambda}(ba^{3}) \, (2ab + ba)$, hence by Lemma~\ref{lambda1b} we
get ${\gamma}(ba^{2}ba^{2}) = 3$.

Let now $m \geq 2$. Arguing by induction on $m$, we assume that ${\gamma}({(ba^{2}ba^{2k})}^{m-1}b)$ is odd. Using \eqref{multilambdashuffle}
we collect all factors of ${(ba^{2}ba^{2k})}^{m}b$ of multi-degree $\big( (m-1)(2k+2), \, 2(m-1)+1 \bigr)$.
All such factors contribute
\begin{equation}
{\lambda}({(ba^{2}ba^{2k})}^{m-1}b) \cdot \{ a^{2}ba^{2k}b \, - \, a^{2}b \sh a^{2k}b \, + \, a^{2k}ba^{2}b \},
\end{equation}
hence we have to show that the non zero coefficients of the polynomial $P \; = \; a^{2}ba^{2k}b \, - \, a^{2}b \sh a^{2k}b \, + \, a^{2k}ba^{2}b$
can not all be even. It is enough to check this on its reversal $\widetilde{P}$ and furthermore on $\widetilde{P} \, \rhd \, b$.
By \eqref{Preversalshuffle} and Proposition~\ref{minhpetitot} we get
$\widetilde{P} \, \rhd \, b \; = \; a^{2k}ba^{2} \, - \, a^{2} \sh ba^{2k} \, - \, a^{2k} \sh ba^{2} \, + \, a^{2}ba^{2k}$.
By Lemma~\ref{abshuffle} $\big( a^{2k} \sh ba^{2}, \, a^{2k-1}ba^{3} \bigr) \, = \, \binom{2+1}{1} = 3$, an odd number, whereas
$\big( a^{2} \sh ba^{2k}, \, a^{2k-1}ba^{3} \bigr) \, = \, 0$, since $k \geq 2$.
Hence $(\widetilde{P} \, \rhd \, b, \, a^{2k-1}ba^{3}) = -3$ and the proof is completed.
\end{proof}

\bigskip

%%%%%%%%%%%%%%%%%%%%%%%%%%% aub --- avb %%%%%%%%%%%%%%%%%%%%%%%%%%%%%%%%%%%%%%%%%%%%%%%%%%%%%%%%%%%%%%%%%%%%%%%%%%%%%%%%%%%%%%%%%%%%%%%%%%%%%%%%%%%%%%
%%%%%%%%%%%%%%%%%%%%%%%%%%%%%%%%%%%%%%%%%%%%%%%%%%%%%%%%%%%%%%%%%%%%%%%%%%%%%%%%%%%%%%%%%%%%%%%%%%%%%%%%%%%%%%%%%%%%%%%%%%%%%%%%%%%%%%%%%%%%%%%%%%%%%%

\section{The case $\mathbf{aub \, \equiv \, \pm \, avb}.$}

For the plus case our objective is to show that $u = v$, whereas in the minus case we must get a contradiction. Starting from the equality ${\lambda}(aub) \: = \: \pm \, {\lambda}(avb)$ and using \eqref{defrecl*} we obtain
\begin{equation}
\bigl( {\lambda}(au) \; \mp \; {\lambda}(av) \bigr) \, b \: = \:
\bigl( {\lambda}(ub) \; \mp \; {\lambda}(vb) \bigr) \, a.
\end{equation}
It follows that
\begin{equation}
  au \: \equiv \: \pm \, av \qquad \mbox{and} \qquad ub \: \equiv \: \, \pm vb \, . \label{aubequiv}
\end{equation}
Clearly $aub \, \not \equiv \, 0$, hence we can not have $au \, \equiv \, ub \equiv \, 0$.
Set $|u| = |v| = r$.

\medskip

{$\mathbf{A. \; aub \, \equiv \, avb}$.}
Suppose that $r$ is odd. Assume that $au \, \not \equiv \, 0$.
Then $|au| = |av| = r + 1$ and $r + 1$ is even, so by \eqref{aubequiv} (with the plus sign) and our induction hypothesis
we get $au = av$, hence $u = v$ as required. If $au \, \equiv \, 0$ then we necessarily have $ub \, \not \equiv \, 0$
and reobtain $u = v$.

Let now $r$ be even. Assume that $au \, \equiv \, 0$. Then
$r + 1$ is odd, so in view of \eqref{aubequiv} our induction hypothesis yields
$au = av = a^{r+1}$, so that $u = v = a^{r}$. Similarly, if $ub \, \equiv \, 0$ we obtain $u = v = b^{r}$.

To conclude this case suppose that both $au \not \equiv 0$ and $ub \not \equiv 0$. By our induction hypothesis
we obtain
\begin{equation} \label{aubavb}
   \begin{cases}
     \, au = av & \text{or} \quad  au = \tilde{v}a \\
     \, ub = vb & \text{or} \quad ub = b{\tilde{v}} .
   \end{cases}
\end{equation}
If at least one of the equalities $au = av$ or $ub = vb$ holds we immediately get $u = v$. Therefore, the only non trivial subcase in \eqref{aubavb}
is the combination
\begin{equation}
  au \: = \:  {\tilde{v}}a \qquad \mbox{and} \qquad ub \: = \: b{\tilde{v}} \, . \label{aubtild}
\end{equation}
Then we obtain
\begin{eqnarray}
aub & = & (au)b \, = \, ({\tilde{v}}a)b \, = \, {\tilde{v}}ab \\
    & = & a(ub) \, = \, a(b{\tilde{v}}) \, = \, ab{\tilde{v}}.
\end{eqnarray}
It follows that $ab{\tilde{v}} \, = \, {\tilde{v}}ab$, i.e., the words $ab$ and ${\tilde{v}}$ commute, hence
by Lemma~\ref{2thL+S} they are both powers of the same word. Since $ab$ is primitive we get
${\tilde{v}} \, = \, {(ab)}^{n}$, hence $v \, = \, {(ba)}^{n}$ for some $n \geq 1$. In view of \eqref{aubtild} we get
$au \, = \, {(ab)}^{n}a \, = \, ab{(ab)}^{n-1}a$. It follows that $u \, = \, b{(ab)}^{n-1}a \, = \, {(ba)}^{n}$ and
therefore $u = v$, as required.

\medskip

{$\mathbf{B. \; aub \, \equiv \, - \, avb}$.}
Suppose that $r$ is even. Without loss of generality we may assume that $au \, \not \equiv 0$. Then we can not have $au \, \equiv \, - \, av$ due
to our induction hypothesis.

Suppose that $r$ is odd. If both $au \, \not \equiv \, 0$ and $ub \, \not \equiv \, 0$ then by \eqref{aubequiv} (with the minus sign) and
our induction hypothesis we reobtain \eqref{aubtild}. Following the lines of Case A we reobtain
${\tilde{v}} \, = \, {(ab)}^{n}$ for some $n \geq 1$. Then $|v| = 2n$ which contradicts our assumption that $r$ is odd.

Suppose finally that, without loss of generality, $au \, \not \equiv \, 0$ and $ub \, \equiv \, 0$. On the one hand \eqref{aubequiv} and our
induction hypothesis yield $au \, = \, {\tilde{v}}a$, and on the other both $ub$ and $vb$ are palindromes of even length.
Then there exists $s \in A^{*}$ such that $u \, = \, sa$, $v \, = \, {\tilde{s}}a$,
$sab \, = \, ba{\tilde{s}}$ and ${\tilde{s}}ab \, = \, bas$. Combining all these we obtain
\begin{equation}
sabab \, = \, (sab)ab \, = \, (ba{\tilde{s}})ab \, = \, ba({\tilde{s}}ab) \, = \, ba(bas) \, = \, babas .
\end{equation}
It follows that $(baba)s \, = \, s(abab)$, so by Lemma~\ref{1thL+S} there exist $p, q \in A^{*}$ such that
$baba \, = \, pq$, $abab \, = \, qp$ and $s \, = p{(abab)}^{n}$ for some integer $n \geq 0$. There are two possibilities
for the factor $p$; either $p = b$ or $p = bab$. In both cases $|s|$ - and consequently $|sab|$ - is an odd positive integer,
contradicting the fact that $ub$ is a palindrome of even length.

\bigskip

%%%%%%%%%%%%%%%%%%%%%%%%%%%%%%%%%%%%%%%%%%%% a^{k}buba^{l} = +- a^{k}bvba^{l} %%%%%%%%%%%%%%%%%%%%%%%%%%%%%%%%%%%%%%%%%%%%%%%%%%%%%%%%%%%%%%%%%%%%%%%%%
%%%%%%%%%%%%%%%%%%%%%%%%%%%%%%%%%%%%%%%%%%%%%%%%%%%%%%%%%%%%%%%%%%%%%%%%%%%%%%%%%%%%%%%%%%%%%%%%%%%%%%%%%%%%%%%%%%%%%%%%%%%%%%%%%%%%%%%%%%%%%%%%%%%%%%%

\section{The case $\mathbf{a^{k}buba^{l} \, \equiv \, \pm \, a^{k}bvba^{l} \, \not \equiv \, 0}.$}

Let $w_{1} \, = \, a^{k}buba^{l}$, $w_{2} \, = \, a^{k}bvba^{l}$ and $|u| = |v| = r$. Multiplying both sides of \eqref{k+l} by $b$ we obtain
\begin{equation} \label{51}
bub \, \equiv \, \pm bvb.
\end{equation}
Suppose that $k = l$. Then since $w_{1}, \, w_{2} \not \equiv \, 0$, we either get $u = v = b^{r}$ with $r$ odd,
or $bub, \, bvb \, \not \equiv \, 0$. In the former case $w_{1} \, \equiv \, w_{2}$ immediately yields $u = v$, hence $w_{1} = w_{2}$, as required.
On the other hand $w_{1} \, \equiv \, - \, w_{2}$ yields $2 \, w_{1} \, \equiv \, 0$, a contradiction.
For the latter, if $w_{1} \, \equiv \, w_{2}$ and $r$ is odd then our induction hypothesis implies $u = v$, thus $w_{1} = w_{2}$, whereas
if $r$ is even we might also get $v = \tilde{u}$, thus $w_{2} = \widetilde{w_{1}}$. On the other hand, when $w_{1} \, \equiv \, - \, w_{2}$ we necessarily get $r$ even and $v = \tilde{u}$ reobtaining $w_{2} = \widetilde{w_{1}}$, as required.

Suppose that $k \neq l$. Without loss of generality we may assume that $k < l$. Then \eqref{k} and \eqref{k+1} respectively yield
\begin{eqnarray}
uba^{l} & \equiv & \pm vba^{l} \label{52} \\
uba^{l-1} & \equiv & \pm vba^{l-1}. \label{53}
\end{eqnarray}
Suppose that $bub \, \equiv \, 0 \, \equiv \, bvb$.
If $r$ is odd then $u = v = b^{r}$, so that if $w_{1} \, \equiv \, w_{2}$ we get the trivial solution $w_{1} = w_{2}$,
whereas if $w_{1} \, \equiv \, - \, w_{2}$ we get $2 \, w_{1} \, \equiv \, 0$, a contradiction.

If $r$ is even we have $u = s{\tilde{s}}$ and $v = t{\tilde{t}}$ for $s, \, t \in A^{*}$. Consider the case $w_{1} \, \equiv \, w_{2}$.
If $l$ is odd then $|s{\tilde{s}}ba^{l}|$ is even, so, in view of our induction hypothesis, \eqref{52} (with the plus sign) yields $u = v$.
The same result follows similarly by \eqref{53} when $l$ is odd. For the case $w_{1} \, \equiv \, - \, w_{2}$ we work in the same manner.
If $l$ is odd (resp. even) then $|s{\tilde{s}}ba^{l-1}|$ (resp. $|s{\tilde{s}}ba^{l}|$) is odd, so \eqref{53} (resp. \eqref{52}) with the minus sign
and our induction hypothesis yield a contradiction.

\smallskip

Suppose that $bub \, \not \equiv \, 0$ and $bvb \, \not \equiv \, 0$. Let $w_{1} \, \equiv \, w_{2}$. If $r$ is even then \eqref{51} reads
$bub \, \equiv \, bvb$, so our induction hypothesis yields $u = v$. If $r$ is odd we might have $v = \tilde{u}$.
If $l$ is odd then $|uba^{l}|$ is also odd and \eqref{52} reads $uba^{l} \, \equiv \, \tilde{u}ba^{l}$.
Our induction hypothesis then implies that $u = \tilde{u}$ directly or it yields $a^{l}bu = uba^{l}$. In the latter case Lemma 2.1 gives
$u = a^{l}{(ba^{l})}^{d}$, for some $d \geq 0$, so that we reobtain $u = \tilde{u}$.
If $l$ is even we end up with the same result working similarly with \eqref{53}.

Finally let $w_{1} \, \equiv \, - \, w_{2}$. Then \eqref{51} reads $bub \, \equiv \, -  \, bvb$, so our induction hypothesis implies that
$r$ is necessarily even and $v = \tilde{u}$. If $l$ is even then $|uba^{l}|$ is odd, so that \eqref{52}, which reads
$uba^{l} \, \equiv \, - \, \tilde{u}ba^{l}$ and our induction hypothesis yield a contradiction.
Similarly we get a contradiction with \eqref{53} (with the minus sign) when $l$ is odd.

\bigskip

%%%%%%%%%%%%%%%%%%%%%%%%%%%%%%%%%%%%%%%%%%%%%%%%%%%%%%%%%% aub = +- ava %%%%%%%%%%%%%%%%%%%%%%%%%%%%%%%%%%%%%%%%%%%%%%%%%%%%%%%%%%%%%%%%%%%%%%%%%%%
%%%%%%%%%%%%%%%%%%%%%%%%%%%%%%%%%%%%%%%%%%%%%%%%%%%%%%%%%%%%%%%%%%%%%%%%%%%%%%%%%%%%%%%%%%%%%%%%%%%%%%%%%%%%%%%%%%%%%%%%%%%%%%%%%%%%%%%%%%%%%%%%%%%

\section{The case $\mathbf{aub \, \equiv \, \pm \, ava}.$}

By definition ${\lambda}(aub) = \pm \, {\lambda}(ava)$, hence by \eqref{defrecl*} it follows that
\begin{equation}
{\lambda}(au)b - {\lambda}(ub)a \: = \pm \, \: \bigl( {\lambda}(av)a - {\lambda}(va)a \bigr).
\end{equation}
Then we necessarily obtain ${\lambda}(au) = 0$, so either $u = a^{n}$, for some $n \geq 1$, or ${|u|}_{b} \geq 1$ and
$au$ is a palindrome of even length.

In the former case $aub = a^{n+1}b$ and $ava = a^{k}ba^{n + 1 - k}$, for some integer $k$ such that $1 \leq k \leq n$. By
Lemma~\ref{lambda1b} we get $\displaystyle {\lambda}(ava) = {(-1)}^{k} \binom{n + 1 }{k} {\lambda}(ba^{n + 1})$
and by Lemma~\ref{l*reversal} ${\lambda}(a^{n+1}b) = {(-1)}^{n + 3} {\lambda}(ba^{n + 1})$, so that we finally obtain
$\displaystyle \binom{n + 1}{k} = \pm 1$, a clear contradiction, since $k \neq 0$ and $k \neq {n + 1}$.

Let us now suppose that $au \equiv 0$ and ${|u|}_{b} \geq 1$, i.e.,
\begin{equation} \label{aubavamain}
\fbox{$a^{k}bs{\tilde{s}}ba^{k}b \, \equiv \, \pm \, a^{l}btba^{m}$} \, ,
\end{equation}
for $s, t \in A^{*}$, $k,l$ and $m$ positive integers with $l \leq m$, without loss of generality.

First consider the case ${|t|}_{a} \, = \, 0$.
%\subsection{The case $\mathbf{{|t|}_{a} \, = \, 0}.$}
Counting the number of occurrences of the letter $a$ in \eqref{aubavamain} we get $2k + 2\,{|s|}_{a} \, = \, l + m$, so that $l + m \, \geq \, 2k$.
By Corollary~\ref{ewdw}  we have $e(a^{k}bs{\tilde{s}}ba^{k}b) \, = \, k$, since $bs{\tilde{s}}ba^{k}b \, \not \equiv \, 0$ and
$e(a^{l}btba^{m}) \, = \, l \, + \, m \, - \, 1$, since $btb \, \equiv \, 0$. Then \eqref{aubavamain} yields
$k \; = \; l \, + \, m \, - \, 1$, so that $k + 1 \geq 2k$ and $k = 1$.
It follows that $l = m = 1$, $bs = b^{n}$ and $t = b^{2n+1}$, for some positive integer $n$, so that \eqref{aubavamain} becomes
\begin{equation}
ab^{2n}ab \, \equiv \, \pm \, ab^{2n + 1}a.
\end{equation}
On the one hand $\displaystyle {\lambda}(ab^{2n}ab) = - {\lambda}(b^{2n}ab)a = - {(-1)}^{2n} {{2n + 1} \choose {2n}}{\lambda}(ab^{2n + 1})a
= - (2n + 1){\lambda}(ab^{2n + 1})a$,
and on the other
${\lambda}(ab^{2n + 1}a) = \bigl( {\lambda}(ab^{2n + 1}) - {\lambda}(b^{2n + 1}a) \bigr) a =
                           2 {\lambda}(ab^{2n + 1})a$.
It follows that $2n + 1 = \pm 2$, a clear contradiction.

Therefore, we may assume that ${|t|}_{a} \, \geq \, 1$.
%\subsection{The case $\mathbf{{|t|}_{a} \, \geq \, 1}.$}
Applying Corollary~\ref{ewdw} in \eqref{aubavamain} we get $l + m = k$, since ${|t|}_{b}$ is odd. It follows that $k \geq 2$.

Our first claim is that $l = 1$. Indeed, if $l > 1$ we apply Proposition~\ref{multieqlambdashuffle} for factors of multi-degree $(1, \, 1)$ on \eqref{aubavamain}. Then working modulo $\ker \lambda$ Lemma~\ref{kerlambda} and Proposition~\ref{Ree} imply that
the only contribution will come from the left hand side and it will yield $a^{k-1}bs{\tilde{s}}ba^{k} \, \equiv \, 0$, a contradiction.

Next we claim that $k = 2$. Suppose that $k > 2$, i.e., consider the situation
\begin{equation} \label{987dr}
a^{k}bs{\tilde{s}}ba^{k-1} \, \fbox{$ab$} \, \equiv \, \pm \, \fbox{$ab$} \, tba^{k-1}.
\end{equation}
Applying Proposition~\ref{multieqlambdashuffle} for factors of multi-degree $(1, \, 1)$ again we finally obtain
\begin{equation} \label{592euros}
 - a^{k-1}bs{\tilde{s}}ba^{k} \, \equiv \, \pm tba^{k-1} .
\end{equation}
Both words in \eqref{592euros} are of odd length so the $+$ case, in view of our induction hypothesis, yields a contradiction.
For the $-$ case our induction hypothesis implies that either $tba^{k-1} = a^{k-1}bs{\tilde{s}}ba^{k}$ or $tba^{k-1} = a^{k}bs{\tilde{s}}ba^{k-1}$.
The former clearly can not hold, whereas the latter yields $t = a^{k}bs{\tilde{s}}$. Thus going back to \eqref{987dr} we must exclude the case where
\begin{equation}
a^{k}bs{\tilde{s}}ba^{k}b \, \equiv \, - \,aba^{k}bs{\tilde{s}}ba^{k-1}, \quad k > 2 \, .
\end{equation}
Observe that $a^{k}bs{\tilde{s}}ba^{k}b \, \equiv \, ba^{k}bs{\tilde{s}}ba^{k}$. We clearly have $ba^{k}bs{\tilde{s}}b \, \not \equiv \, 0$,
therefore by \eqref{k+l} we get
$\displaystyle {\lambda}(ba^{k}bs{\tilde{s}}b) = - {(-1)}^{1} {k \choose 1} {\lambda}(ba^{k}bs{\tilde{s}}b)$.
It follows that $k = 1$, which contradicts the fact that $k > 2$.

\subsection{The case $\mathbf{a^{2}bs{\tilde{s}}ba^{2}b \, \equiv \, \pm \, abtba}.$}

Suppose that $t = ayb$, for some $y \in A^{*}$ (the case $t = bya$ is dealt similarly), i.e., consider the equation
\begin{equation}
a^{2}bs{\tilde{s}}ba^{2}b \, \equiv \, \pm \, abay \, \fbox{$b^{2}a$} \, .
\end{equation}
Applying Proposition~\ref{multieqlambdashuffle} for all factors of multi-degree $(1, \, 2)$ and then working modulo $\ker \lambda$ we finally
obtain $abay \, \equiv \, 0$, which contradicts the fact that ${|y|}_{b}$ is even.

We proceed in the same manner with the case $t = byb$, i.e.,
\begin{equation} \label{joconda}
a^{2}bs{\tilde{s}}ba^{2}b \, \equiv \, \pm \, \fbox{$ab^{2}$} \, y \, \fbox{$b^{2}a$}
\end{equation}
and get $yb^{2}a + \tilde{y}b^{2}a \equiv 0$, as we have no contribution from the left hand side of \eqref{joconda}.
Both these words are of even length and since ${|y|}_{b}$ is odd in this case, they are $\not \equiv 0$.
Then our induction hypothesis yields $ab^{2}y = yb^{2}a$.
By Lemma~\ref{1thL+S} there exist $u, v \in A^{*}$ such that $ab^{2} = uv$, $b^{2}a = vu$ and $y = u{(b^{2}a)}^{n}$,
for some $n \geq 0$. It follows that $u = a$ and $v = b^{2}$, so that
$y = a{(b^{2}a)}^{n}$. But then ${|y|}_{b} = 2n$, which contradicts the fact that ${|y|}_{b}$ is odd.

We may therefore assume that
\begin{equation}
a^{2}bs{\tilde{s}}ba^{2}b \, \equiv \, \pm \, abayaba.
\end{equation}
Suppose that $s = \epsilon$, i.e., $a^{2}b^{2}a^{2}b \, \equiv \, \pm \, abababa$. Then
$b^{2}a^{2}b \, \equiv \, \mp \, 2 \, babab$. Since ${\gamma}(b^{2}a^{2}b) = 1$ by Lemma~\ref{gamma1}, this can not hold as it would
imply that $2 \, | \, 1$, a contradiction.

Suppose that $s = bx$, for some $x \in A^{*}$, i.e., $a^{2}b^{2}x{\tilde{x}}b^{2}a^{2}b \, \equiv \, \pm \, abayaba$.
It follows that $b^{2}x{\tilde{x}}b^{2}a^{2}b \, \equiv \, \mp \, 2 \, bayab$. Since ${|y|}_{b}$ is odd $aya \not \equiv 0$, so
by Corollary~\ref{ewdw} $e_{b}(bayab) = 2$. It follows that $e_{b}(b^{2}x{\tilde{x}}b^{2}a^{2}b) = 2$ and this is possible only if
$x{\tilde{x}}b^{2}a^{2} \equiv 0$, i.e., the word $x{\tilde{x}}b^{2}a^{2}$ is a palindrome of even length.
Thus $a^{2}b^{2}x{\tilde{x}} = x{\tilde{x}}b^{2}a^{2}$ and by Lemma~\ref{1thL+S} we get $x{\tilde{x}} = a^{2}{(b^{2}a^{2})}^{n}$, for some $n \geq 0$.
Then we obtain ${(b^{2}a^{2})}^{n + 2}b \, \equiv \, \mp \, 2 \, bayab$.
By Lemma~\ref{gamma1} we have ${\gamma} \bigl( {(b^{2}a^{2})}^{n + 2}b \bigr) = 1$, hence we reobtain the contradiction
$2 \, | \, 1$.

We may therefore suppose that $s = a^{k}bx$, for some $x \in A^{*}$, so it remains to check the validity of
\begin{equation} \label{aubavamain2}
\fbox{$ \, a^{2}ba^{k}bx{\tilde{x}}ba^{k}ba^{2}b \, \equiv \, \pm \, aba^{l}byba^{m}ba \, $} \, ,
\end{equation}
where $k,l,m \geq 1$ and $l \leq m$ without loss of generality.

Suppose first that ${|y|}_{a} = 0$, i.e., $byb = b^{2n + 1}$ for $n \geq 1$. We obtain
$ba^{k}bx{\tilde{x}}ba^{k}ba^{2}b \, \equiv \, \mp \, 2 \, ba^{l}b^{2n+1}a^{m}b$ and consequently
$a^{k}bx{\tilde{x}}ba^{k}ba^{2} \, \equiv \, \mp \, 2 \, a^{l}b^{2n+1}a^{m}$. By Corollary~\ref{ewdw}
$e(a^{l}b^{2n+1}a^{m}) = l + m -1$ and $e(a^{k}bx{\tilde{x}}ba^{k}ba^{2}) = k + 2$, since clearly $bx{\tilde{x}}ba^{k}b \not \equiv 0$.
It follows that $k + 2 = l + m - 1$, i.e., $l + m = k + 3$.
On the other hand, counting occurrences of the letter $a$ we obtain $2k + 2{|x|}_{a} + 2 \, = \, l + m$.
It follows that $k = 1$ and ${|x|}_{a} = 0$, so that $bx{\tilde{x}}b = b^{2n}$ and $ab^{2n}aba^{2} \, \equiv \, \mp \, 2 \, a^{l}b^{2n + 1}a^{m}$,
where $l + m = 4$ and $1 \leq l \leq m$.
Then by Corollary~\ref{Pi}, and in particular \eqref{k+l} and \eqref{k+l-1}, there exists $ q \in \mathds Z$ such that
$- \, \binom{1 + 2}{1} \, b^{2n}ab  \, \equiv \, \pm \, 2 q \, ab^{2n + 1}$. By Lemma~\ref{lambda1b} it follows that
$-3(2n + 1) ab^{2n + 1} \, \equiv \, \pm \, 2 q \, ab^{2n + 1}$ which is a contradiction since it would mean
that $2 \, | \, 3(2n + 1)$.

\medskip

Therefore we may suppose that ${|y|}_{a} \geq 1$. Since ${|y|}_{b}$ is odd Corollary~\ref{ewdw} now yields
\begin{equation} \label{eqklm}
k \, + \, 2 \: = \: l \, + \, m \, .
\end{equation}

\subsubsection{The case $\mathbf{l \, = \, 1}.$}

By \eqref{eqklm} we have $m = k + 1$, i.e., we consider the equation
\begin{equation} \label{aubavamain3}
a^{2}ba^{k}bx{\tilde{x}}ba^{k}ba^{2}b \, \equiv \, \pm \, \fbox{$abab$} \, yba^{k + 1}ba.
\end{equation}
We apply Proposition~\ref{multieqlambdashuffle} with factors $u$ of multi-degree $(2,2)$.
Since ${\lambda}(abab) = - 2 {\lambda}(a^{2}b^{2}) = {\lambda}(baba)$ the term ${\lambda}(a^{2}b^{2})$ will appear as a common factor and
therefore may be canceled. Furthermore, the only contribution modulo $\ker {\lambda}$ will come from the right hand side in \eqref{aubavamain3}
since $ba^{2}b \equiv 0$ and it will finally lead to $yba^{k + 1}ba \, \equiv \, 0$, which contradicts the fact that ${|y|}_{b}$ is odd.

\subsubsection{The case $\mathbf{l \, = \, 2 \: \, \text{and} \; \, k \, = \, 2 }.$}

By \eqref{eqklm} we also get $m = 2$, hence we are dealing with the equation
\begin{equation} \label{aubavamain4}
\fbox{$a^{2}ba^{2}$} \, bx{\tilde{x}}ba^{2}ba^{2}b \, \equiv \, \pm \, aba^{2}byba^{2}ba.
\end{equation}
We apply Proposition~\ref{multieqlambdashuffle} with factors $u$ of multi-degree $(4,1)$.
Then the term ${\lambda}(a^{4}b)$ will appear as a common factor and therefore
may be canceled. Furthermore, the only contribution modulo $\ker {\lambda}$ will come from the left hand side and will yield
$bx{\tilde{x}}ba^{2}ba^{2}b \, \equiv \, 0$, i.e., $x{\tilde{x}}ba^{2}ba^{2} \, \equiv \, 0$.
Therefore we obtain $a^{2}ba^{2}bx{\tilde{x}} \, = \, x{\tilde{x}}ba^{2}ba^{2}$, so by Lemma~\ref{1thL+S} there exist words $u, v \in A^{*}$
such that $a^{2}ba^{2}b = uv$, $ba^{2}ba^{2} = vu$ and $x{\tilde{x}} = u{(ba^{2}ba^{2})}^{n}$, for some $n \geq 0$.
Two possibilities arise; either $u = a^{2}ba^{2}$ and $v = b$ or $u = a^{2}$ and $v = ba^{2}b$. The former one
can not occur since $x{\tilde{x}}$ would be of odd length. The latter one yields
$x{\tilde{x}} = a^{2}{(ba^{2}ba^{2})}^{n}$, which is indeed a palindrome of even length.
Therefore the word on the left hand side of \eqref{aubavamain4} is in fact equal to
$a^{2}{(ba^{2})}^{2n + 4}b = {(a^{2}b)}^{2n + 5}$, where $n \geq 0$. Then \eqref{aubavamain4} may be read as
\begin{equation} \label{aubavamain5}
 \fbox{$a^{2}b$} \, {(a^{2}b)}^{2n+3} \, \fbox{$a^{2}b$} \, \equiv \, \pm \, \fbox{$aba$} \, abyba \, \fbox{$aba$} \, .
\end{equation}
Applying Proposition~\ref{multieqlambdashuffle} with factors $u$ of multi-degree $(2,1)$ we finally obtain
\begin{equation} \label{aubavamain51}
\left( {(a^{2}b)}^{2n+4} \, + \, {(ba^{2})}^{2n+4} \right) \, \equiv \, \mp \, 2 \, \left( abyba^{2}ba \, + \, ab{\tilde{y}}ba^{2}ba \right) \! .
\end{equation}
The left hand side of \eqref{aubavamain51} is equal to zero, hence it follows that $abyba^{2}ba \, \equiv \, - \, ab{\tilde{y}}ba^{2}ba$.

Suppose that $abyba^{2}ba \, \not \equiv \, 0$. Then since both words are of even length our induction
hypothesis yields $abyba^{2}ba \, = \, aba^{2}byba$, thus $yba^{2} \, = \, a^{2}by$. By Lemma~\ref{1thL+S} there exist words
$u, v \in A^{*}$ such that $a^{2}b = uv$, $ba^{2} = vu$ and $y = u{(ba^{2})}^{n}$, for some $n \geq 0$. Then
we necessarily have $u = a^{2}$ and $v = b$ so that $y = a^{2}{(ba^{2})}^{d}$, for some $d \geq 0$.
Going back to \eqref{aubavamain5} we see that $d = 2n+1$. But then $abyba^{2}ba = aba^{2}{(ba^{2})}^{2n+1}ba^{2}ba = a{(ba^{2})}^{2n+3}ba$ and
the latter is a palindrome of even length contradicting our assumption that $abyba^{2}ba \, \not \equiv \, 0$.

We are thus left with the case $abyba^{2}ba \, \equiv \, 0 \, \equiv \, ab{\tilde{y}}ba^{2}ba$. Then
both $yba^{2}$ and ${\tilde{y}}ba^{2}$ are palindromes of even length, i.e.,
$yba^{2} = a^{2}b{\tilde{y}}$ and ${\tilde{y}}ba^{2} = a^{2}by$. But then on the one hand
$a^{2}b{\tilde{y}}ba^{2} = (a^{2}b{\tilde{y}})ba^{2} = (yba^{2})ba^{2} = yba^{2}ba^{2}$
and on the other
$a^{2}b{\tilde{y}}ba^{2} = a^{2}b({\tilde{y}}ba^{2}) = a^{2}b(a^{2}by) = a^{2}ba^{2}by$.
It follows that $a^{2}ba^{2}by = yba^{2}ba^{2}$ so once more by Lemma~\ref{1thL+S} there exist $u, v \in A^{*}$ such that
$a^{2}ba^{2}b = uv$, $ba^{2}ba^{2} = vu$ and $y = u{(ba^{2}ba^{2})}^{d}$, for some $d \geq 0$. This time
the only possibility is $u = a^{2}ba^{2}$ and $v = b$ since we demand that ${|y|}_{b}$ is odd.
So we obtain $y = a^{2}ba^{2}{(ba^{2}ba^{2})}^{d}$, hence $y$ is a palindrome of odd length.
Then clearly $d = n$ and \eqref{aubavamain5} becomes ${(a^{2}b)}^{2n+5} \, \equiv \, \pm \, a{(ba^{2})}^{2n+4}ba$, which in turn yields
${(ba^{2})}^{2n+4}b \, \equiv \, \mp \, 2 \, {(ba^{2})}^{2n+4}b$. In any case it follows that
${(ba^{2})}^{2n+4}b \, \equiv \, 0$, which can not hold since the corresponding length is odd.

\subsubsection{The case $\mathbf{l \, \geq \, 2 \: \mbox{and} \: k \, > \, 2 }.$}

Here we are checking upon the equation
\begin{equation} \label{aubavamain6}
\fbox{$a^{2}ba^{k}$} \, bx{\tilde{x}}ba^{k}ba^{2}b \, \equiv \, \pm \, aba^{l}byba^{m}ba.
\end{equation}
We apply Proposition~\ref{multieqlambdashuffle} with factors $u$ of multi-degree $(k+2, \, 1)$. The term ${\lambda}(ba^{k+2})$ will appear
as a common factor and therefore may be canceled. Due to \eqref{eqklm} we get $k + 2 > l + 1$ and $k + 2 > m + 1$, hence the only contribution
modulo $\ker {\lambda}$ will come from the left hand side and will yield $bx{\tilde{x}}ba^{k}ba^{2}b \, \equiv \, 0$. It follows that
the word $x{\tilde{x}}ba^{k}ba^{2}$ is a palindrome of even length, therefore $k$ must be even and $k \geq 4$. We obtain
$a^{2}ba^{k}bx{\tilde{x}} = x{\tilde{x}}ba^{k}ba^{2}$, so by Lemma~\ref{1thL+S} there exist $u, v \in A^{*}$ such that
$a^{2}ba^{k}b = uv$, $ba^{k}ba^{2} = vu$ and $x{\tilde{x}} = u{(ba^{k}ba^{2})}^{n}$, for some $n \geq 0$.
Since ${|u|}_{b}$ must be even we obtain $u = a^{2}$, $v = ba^{k}b$ and $x{\tilde{x}} = a^{2}{(ba^{k}ba^{2})}^{n}$.
Then \eqref{aubavamain6} becomes
\begin{equation}
a^{2}{(ba^{k}ba^{2})}^{n+2}b \, \equiv \, \pm \, aba^{l}byba^{m}ba,
\end{equation}
which further yields
${(ba^{k}ba^{2})}^{n+2}b \, \equiv \, \mp \, 2 \, ba^{l}byba^{m}b$. Passing to reversals by Lemma~\ref{l*reversal} we get
\begin{equation}
{(ba^{2}ba^{k})}^{n+2}b \, \equiv \, \mp \, 2 \, ba^{m}b{\tilde{y}}ba^{l}b.
\end{equation}
It follows that ${\gamma} \big(  {(ba^{2}ba^{k})}^{n+2}b \bigr) \, = \, 2 \, {\gamma}(ba^{m}b{\tilde{y}}ba^{l}b)$.
But the latter contradicts Lemma~\ref{gammaodd} which states that ${\gamma} \bigl( {(ba^{2}ba^{k})}^{n+2}b \bigr)$ is odd.

\bigskip

%%%%%%%%%%%%%%%%%%%%%%%%%%%%%%%%%%%%%%%%%%%%%% a^{k}bx{\tilde{x}}ba^{k+1} = +- a^{k}byba^{k} %%%%%%%%%%%%%%%%%%%%%%%%%%%%%%%%%%%%%%%%%%%%%%%%%%%%%%%%%%
%%%%%%%%%%%%%%%%%%%%%%%%%%%%%%%%%%%%%%%%%%%%%%%%%%%%%%%%%%%%%%%%%%%%%%%%%%%%%%%%%%%%%%%%%%%%%%%%%%%%%%%%%%%%%%%%%%%%%%%%%%%%%%%%%%%%%%%%%%%%%%%%%%%%%%%

\section{The case $\mathbf{a^{k}bx{\tilde{x}}ba^{k+1} \, \equiv \, \pm \, a^{k}byba^{k}}.$}

Set $w_{1} = a^{k}bx{\tilde{x}}ba^{k+1}$ and $w_{2} = a^{k}byba^{k}$. We assume, for the sake of contradiction,
that $w_{1} \, \sim \, w_{2}$. Recall that this means that there exist ${\eta}_{1}, {\eta}_{2} \in {\mathds Z}^{*}$ such that
${\eta}_{1} \, w_{1} \, \equiv \, {\eta}_{2} \, w_{2}$. In certain cases we will be able to show that this is impossible; in certain
others we will only reach our original goal, i.e., that $w_{1} \, \not \equiv \pm \, w_{2}$. The reason for aiming for the
impossibility of the more difficult equation $w_{1} \, \sim \, w_{2}$, will be revealed later on in \S 8.

\smallskip

Let us first deal with the case ${|x|}_{b} = 0$. Then there exists $l \geq 0$ such that
\begin{equation}
\fbox{$a^{k}b$} \, a^{2l}ba^{k+1}  \, \equiv \, \pm \, \fbox{$a^{k}b$} \, a^{2l+1} \fbox{$ba^{k}$} \, .
\end{equation}
We apply Proposition~\ref{multieqlambdashuffle} with factors $u$ of multi-degree $(k, \, 1)$ and modulo $\ker \lambda$ we finally obtain
$a^{2l}ba^{k+1}  \, \equiv \, \pm \, 2 \, a^{2l+1}ba^{k}$.
Then by Lemma~\ref{lambda1b} we obtain $\binom{k+2l+1}{k+1} \: = \: \mp \, 2 \, \binom{k+2l+1}{k}$, which in turn
yields $2l+1 \: = \: \mp \, 2(k+1)$, a contradiction.

Note that in general the equation $a^{k}ba^{2l}ba^{k+1}  \, \sim \, a^{k}ba^{2l+1}ba^{k}$
might hold. For example if $k = l = 1$, by Proposition~\ref{lambdapascal} one can show that
${\lambda}(aba^{2}ba^{2}) \: = \: 3 \, {\lambda}(ba^{3})ba^{2} \: + \: 6 \, {\lambda}(ba^{4})ba$, whereas
${\lambda}(aba^{3}ba) \: = \: -4 \, {\lambda}(ba^{3})ba^{2} \: - \: 8 \, {\lambda}(ba^{4})ba$. It follows
that $4 \, aba^{2}ba^{2} \: \equiv \: - \, 3 \, aba^{3}ba$.

\smallskip

In the rest of the section we deal with the case ${|x|}_{b} \geq 1$.

\subsection{The case $\mathbf{y \, = \, btb}.$}

We may write $x \, = \, a^{l}bs$ for some $s \in A^{*}$ and $l \geq 0$, i.e., we are dealing with the equation
\begin{equation} \label{711}
{\eta}_{1} \, a^{k}ba^{l}bs{\tilde{s}}b \, \fbox{$a^{l}ba^{k+1}$} \, \equiv \, {\eta}_{2} \, a^{k}b^{2}tb^{2}a^{k}.
\end{equation}
Applying Proposition~\ref{multieqlambdashuffle} for factors of multi-degree $(k+l+1, \, 1)$ we finally obtain
$bs{\tilde{s}}ba^{l}ba^{k} \, \equiv \, 0$, a contradiction.

\bigskip

\subsection{The case $\mathbf{x \, = \, a^{l}bs, \quad y \, = \, a^{m}btba^{n}}.$}

We are dealing with the comparison
${\eta}_{1} \, a^{k}ba^{l}bs{\tilde{s}}ba^{l}ba^{k+1} \, \equiv \, {\eta}_{2} \, a^{k}ba^{m}btba^{n}ba^{k}$,
where $l,m,n \geq 1$ and $m \leq n$ without loss of generality.
By Proposition~\ref{multieqlambdashuffle} for factors of multi-degree $(0, \, 1)$ - i.e., factoring out the letter $b$ - we obtain
\begin{equation} \label{trio}
 {\eta}_{1} \, \left \{ \begin{array}{c}
  + \, a^{k} \: \sh \:   a^{l}bs{\tilde{s}}ba^{l}ba^{k+1}  \\
 + \, {(-1)}^{l+1} \, a^{l}ba^{k} \: \sh \: s{\tilde{s}}ba^{l}ba^{k+1} \\
            \vdots                \\
 + \, {(-1)}^{l}  \, s{\tilde{s}}ba^{l}ba^{k} \: \sh \: a^{l}ba^{k+1} \\
 - \, a^{l}bs{\tilde{s}}ba^{l}ba^{k} \: \sh \: a^{k+1}
      \end{array} \right \} \: = \: {\eta}_{2} \,
 \left \{ \begin{array}{c}
 + \, a^{k} \: \sh \: \fbox{$a^{m}btba^{n}$} \, ba^{k} \\
 + \, {(-1)}^{m+1} \, a^{m}ba^{k} \: \sh \: tba^{n}ba^{k}  \\
           \vdots   \\
 + \, {(-1)}^{n+1} \, \tilde{t}ba^{m}ba^{k} \: \sh \: a^{n}ba^{k}  \\
 + \, \fbox{$a^{n}b{\tilde{t}}ba^{m}$} \, ba^{k} \: \sh \: a^{k}
    \end{array} \right  \} .
\end{equation}
Let $u = a^{m}btba^{n}$. Since $|u|$ is odd we know that ${\lambda}(u) \, = \, {\lambda}({\tilde{u}}) \, \neq 0$,
therefore there exists a non zero integer $r$ and a Lie polynomial $P$ such that
$(P, \, u) \, = \, (P, \, {\tilde{u}}) \, = \, r$. We act by $\triangleright \, P$ on both sides of
\eqref{trio} using Proposition~\ref{minhpetitot}.
On the first and the last line of the right hand side in \eqref{trio} we obtain
\[ (a^{k} \: \sh \: uba^{k}) \, \triangleright \, P \: = \:
   (a^{k} \: \sh \: {\tilde{u}}ba^{k}) \, \triangleright \, P \: = \:
    r \, (a^{k} \: \sh \: ba^{k}). \]
The words $u$ and $a^{l}bs{\tilde{s}}ba^{l}b$ have the same length but not the same multi-degree
since ${|a^{l}bs{\tilde{s}}ba^{l}b|}_{a} \, = \, {|u|}_{a} - 1$. It follows that
$(P, \, a^{l}bs{\tilde{s}}ba^{l}b) = 0$, therefore, there is no contribution from the action
$\triangleright \, P$ to the first and the last line of the left hand side in \eqref{trio}.
Set $(P, \, s{\tilde{s}}ba^{l}ba^{l+1}) \, = \, r_{1}$, $(P, \, tba^{n}ba^{m}) \, = \, r_{2}$ and
$(P, \, {\tilde{t}}ba^{m}ba^{n}) \, = \, r_{3}$ and set $e_{i}^{\, j}$ to be equal to $1$, when $i \geq j$,
and $0$ otherwise. Then the action of $P$ on the second and on the penultimate line in both sides of \eqref{trio}
is given by the equations
\begin{equation}
\, \left \{ \begin{array}{rccl}
(a^{l}ba^{k} \: \sh \: s{\tilde{s}}ba^{l}ba^{k+1}) & \triangleright \, P & = &
 e_{k}^{\,l} \, r_{1} \, (a^{l}ba^{k} \: \sh \: a^{k-l}), \\
(s{\tilde{s}}ba^{l}ba^{k}  \: \sh \: a^{l}ba^{k}) & \triangleright \, P & = &
 e_{k}^{\,l+1} \, r_{1} \, (a^{k-l-1}  \: \sh \: a^{l}ba^{k+1}), \\
(a^{m}ba^{k} \: \sh \: tba^{n}ba^{k}) & \triangleright \, P & = &
 e_{k}^{\,m} \, r_{2} \, (a^{m}ba^{k} \: \sh \: a^{k-m}), \\
 ({\tilde{t}}ba^{m}ba^{n}  \: \sh \: a^{n}ba^{k}) & \triangleright \, P & = &
 e_{k}^{\,n} \, r_{3} \, (a^{k-n}  \: \sh \: a^{n}ba^{k}).
\end{array} \right .
\end{equation}
Collecting all terms from both sides we obtain
\begin{equation}
{\eta}_{1} \, \left \{ \begin{array}{c}
{(-1)}^{l+1}e_{k}^{\,l}r_{1} \, (a^{l}ba^{k} \: \sh \: a^{k-l}) \: + \: \\
   {(-1)}^{l}e_{k}^{\,l+1}r_{1} \, (a^{k-l-1} \: \sh \: a^{l}ba^{k+1}) \end{array} \right \} \: = \:
{\eta}_{2} \, \left \{ \begin{array}{c} 2r \, (a^{k} \: \sh \: \fbox{b} \, a^{k}) \: + \: \\
   {(-1)}^{m+1}e_{k}^{\,m}r_{2} \, (a^{m}ba^{k} \: \sh \: a^{k-m}) \: + \: \\
   {(-1)}^{n+1}e_{k}^{\,n}r_{3} \, (a^{k-n} \: \sh \: a^{n}ba^{k})
   \end{array} \right \} .  \label{tralala}
\end{equation}
Now observe that all words appearing in the shuffle products in \eqref{tralala} begin with the letter $a$,
except from the word $ba^{2k}$ in the term $(a^{k} \: \sh \: ba^{k})$. Thus if we act by $\triangleright \, b$
on both sides of \eqref{tralala} we obtain $\displaystyle 2{\eta}_{2}r \, ( a^{k} \, \sh \, a^{k}) \, = \, 0$.
Since ${\eta}_{2}, \, r \, \neq \, 0$ we get $\displaystyle \binom{2k}{k} \, = \, 0$, a contradiction.

\subsection{The case $\mathbf{ x = \, a^{l}bs, \quad y \, = \, a^{m}btb}.$}

Suppose that there exist ${\eta}_{1}, {\eta}_{2} \in {\mathds Z}^{*}$ such that
\begin{equation} \label{731}
{\eta}_{1} \, a^{k}ba^{l}bs{\tilde{s}}ba^{l}ba^{k+1} \: \equiv \: {\eta}_{2} \, a^{k}ba^{m}btb^{2}a^{k} \, .
\end{equation}
By \eqref{k+l} and \eqref{k+l-1} we obtain
\begin{equation} \label{732}
{\eta}_{1} \, {(-1)}^{k-1} \, \{ \binom{2k}{k+1} \, - \, \binom{2k}{k} \} \, {\lambda}(aba^{l}bs{\tilde{s}}ba^{l}) \, b \; = \;
{\eta}_{2} \, \binom{2k}{k} \,  {\lambda}(ba^{m}btb^{2}) .
\end{equation}
Let $P$ and $Q$ denote respectively the left and the right hand side of \eqref{732}.
Observe that $e_{b}(P) = 1$, whereas $e_{b}(Q) \geq 2$, a contradiction.

\bigskip

\subsection{The case $\mathbf{x \, = \, bs, \quad y \, = \, a^{l}btba^{m}, \quad l \, \leq \, m}. $}

We begin by examining the validity of the general equation
\begin{equation} \label{filedem}
{\eta}_{1} \, \fbox{$a^{k}b$} \, b \, \underbrace{s{\tilde{s}}b^{2}a^{k+1}} \: \equiv \: {\eta}_{2} \; a^{k}ba^{l}btba^{m}ba^{k} \, ,
\end{equation}
when ${\eta}_{1}, {\eta}_{2}$ are non zero integers.
We apply Proposition~\ref{multieqlambdashuffle} for all factors of multi-degree $(0, \, 1)$ and factoring out the letter $b$
we obtain
\begin{equation} \label{triole}
 {\eta}_{1} \, \left \{ \begin{array}{c}
  + \, a^{k} \: \sh \:   bs{\tilde{s}}b^{2}a^{k+1}  \\
  - \, \fbox{$ba^{k}$} \: \sh \: s{\tilde{s}}b^{2}a^{k+1} \\
            \vdots                \\
  + \, s{\tilde{s}}b^{2}a^{k} \: \sh \: ba^{k+1} \\
  - \, bs{\tilde{s}}b^{2}a^{k} \: \sh \: a^{k+1}
      \end{array} \right \} \: = \: {\eta}_{2} \,
 \left \{ \begin{array}{c}
 + \, a^{k} \: \sh \: a^{l}btba^{m}ba^{k} \\
 + \, {(-1)}^{l+1} \, a^{l}ba^{k} \: \sh \: tba^{m}ba^{k}  \\
           \vdots   \\
 + \, {(-1)}^{m+1} \, \tilde{t}ba^{l}ba^{k} \: \sh \: a^{m}ba^{k}  \\
 + \, a^{m}b{\tilde{t}}ba^{l}ba^{k} \: \sh \: a^{k}
    \end{array} \right  \} .
\end{equation}
We act by $[a^{k}b] \, \rhd$ on \eqref{triole} using Proposition~\ref{minhpetitot}. In view of Proposition~\ref{Ree}, working modulo
$\ker \lambda$ we finally obtain $s{\tilde{s}}b^{2}a^{k+1} \, \equiv \, 0$. Thus $k$ must be odd and $s{\tilde{s}}b^{2}a^{k+1}$ must be
a palindrome, i.e., $a^{k+1}b^{2}s{\tilde{s}} \, = \, s{\tilde{s}}b^{2}a^{k+1}$ and Lemma~\ref{1thL+S} then yields
$s{\tilde{s}} \, = \, a^{k+1}{(b^{2}a^{k+1})}^{d}$, for some $d \geq 0$.

Suppose that $l < m$. Then \eqref{filedem} becomes
\begin{equation} \label{shfhs}
{\eta}_{1} \, a^{k}b^{2}a^{k+1}{(b^{2}a^{k+1})}^{d} b^{2}a^{k+1} \: \equiv \: {\eta}_{2} \; a^{k}ba^{l}btb \, \fbox{$a^{m}ba^{k}$} \, .
\end{equation}
Applying Proposition~\ref{multieqlambdashuffle} for all factors of multi-degree $(k+m, \, 1)$ on \eqref{shfhs} and then working modulo $\ker \lambda$
with Proposition~\ref{Ree}, we finally obtain $b{\tilde{t}}ba^{l}ba^{k} \, \equiv \, 0$, which is clearly impossible.

Now suppose that $1 < l = m$. Then \eqref{filedem} is written as
\begin{equation} \label{pragmateyths}
{\eta}_{1} \, a^{k}b^{2}a^{k+1}{(b^{2}a^{k+1})}^{d}b^{2}a^{k+1} \: \equiv \:
{\eta}_{2} \, \fbox{$a^{k}ba^{l}$} \, btb \, \fbox{$a^{l}ba^{k}$} \, .
\end{equation}
Once more we apply Proposition~\ref{multieqlambdashuffle} for all factors of multi-degree $(k+l, \, 1)$ on \eqref{pragmateyths}. This time
working modulo $\ker \lambda$ we get $btba^{l}ba^{k} \, + \, b{\tilde{t}}ba^{l}ba^{k} \, \equiv \, 0$.
Set $|btba^{l}ba^{k}| = r$. Since $btba^{l}ba^{k} \, \not \equiv \, 0$ our induction hypothesis yields an immediate contradiction when $r$
is odd, whereas when $r$ is even it gives $btba^{l}ba^{k} = a^{k}ba^{l}btb$, which also clearly can not hold.

It remains to check the case $l = m = 1$. From this point and on we return to our initial objective, i.e., we consider the equation
\begin{equation}
a^{k}b^{2}a^{k+1}{(b^{2}a^{k+1})}^{d}b^{2}a^{k+1} \: \equiv \: \pm \; a^{k}babtbaba^{k} \, ,
\end{equation}
where $d \geq 0$ and $k$ is an odd positive integer.
By \eqref{k} we obtain
\begin{equation} \label{useagainin95}
{\lambda}(ba^{k+1}{(b^{2}a^{k+1})}^{d}b^{2}a^{k+1}) \: = \: \pm \; \{ \, {\lambda}(abtbaba^{k})
\: + \: {\lambda}(a^{k}babtba) \, \} \, .
\end{equation}
Using \eqref{bubal} and \eqref{k+l} to extract the largest powers of $a$ from the right we obtain
\begin{equation} \label{makealiving}
{\lambda}(ba^{k+1}{(b^{2}a^{k+1})}^{d}b^{2}) \: = \: \mp \; (k + 1) \, \{ {\lambda}(btbab) \: + \:
{\lambda}(babtb) \, \} \, .
\end{equation}
Since $ba^{k+1}{(b^{2}a^{k+1})}^{d}b^{2} = b{(a^{k+1}b^{2})}^{d+1}$ and $k$ is odd Lemma~\ref{gamma1} yields
${\gamma} \bigl( b{(a^{k+1}b^{2})}^{d+1} \bigr) = 1$. But then \eqref{makealiving} implies that
$(k + 1) \, | \, 1$, a contradiction.

\smallskip

At this point one should mention that there might exist ${\eta}_{1}$ and ${\eta}_{2} \in {\mathds Z}^{*}$ such that
\begin{equation}
{\eta}_{1} \, a^{k}b^{2}a^{k+1}{(b^{2}a^{k+1})}^{d}b^{2}a^{k+1} \: \equiv \: {\eta}_{2} \, a^{k}babtbaba^{k} \, .
\end{equation}
Indeed, for $k = 1$ and $d = 0$ one can show that $16 \, ab^{2}a^{2}b^{2}a^{2} \, \equiv \, ababababa$ using Proposition~\ref{lambdapascal}.

\bigskip

\subsection{The case $\mathbf{x \, = \, bs, \quad y \, = \, a^{l}btb}.$}

\subsubsection{The case $\mathbf{{|s|}_{a} \, = \, 0}.$}

It means that $s = b^{m}$, for some $m \geq 0$.
Counting the occurrences of the letter $a$ in this particular case we necessarily get $l = 1$, so that
$t = b^{2m}$ and we have to check the validity of the equation
\begin{equation} \label{abcd}
 {\eta}_{1} \, a^{k}b^{2m+4}a^{k+1} \: \equiv \:
 {\eta}_{2} \, a^{k}bab^{2m+3}a^{k} .
\end{equation}
By \eqref{k} we obtain ${\eta}_{1} \, b^{2m+3}a^{k+1} \: \equiv \: {\eta}_{2} \, \bigl( ab^{2m+3}a^{k} \, + \, b^{2m+2}aba^{k} \bigr)$.
%We apply Proposition~\ref{multieqlambdashuffle} for factors of multi-degree $(1, \, 2m+2)$ on \eqref{efgh}. Factoring out the term
%${\lambda}(ab^{2m+2})$ and working modulo $\ker \lambda$ we get $2 {\eta}_{2} \, ba^{k} \, \equiv \, 0$, when $k > 1$ and
%$3 {\eta}_{2} \, ba \, \equiv \, 0$. In both cases we get a contradiction since ${\eta}_{2}$ was assumed to be non zero.
Equating polynomial coefficients of $b$ and passing to reversals we obtain
\begin{equation} \label{ramnous}
{\eta}_{1} \, a^{k+1}b^{2m+2} \: \equiv \: {\eta}_{2} \, a^{k}bab^{2m+1}.
\end{equation}
Unless $k = 1$ and $m = 0$, both words in \eqref{ramnous} are distinct Lyndon words so \eqref{ramnous} contradicts Proposition~\ref{lambdaLyndon}.
If $k = 1$ and $m = 0$ the word $abab$ on the right hand side is not Lyndon. Since $abab \equiv - 2 a^{2}b^{2}$ \eqref{ramnous} in that particular
case yields ${\eta}_{1} = - 2{\eta}_{2}$, so that \eqref{abcd} becomes
\begin{equation} \label{nyxterino}
-2 \, \fbox{$ab$} \, b^{3}a^{2} \, \equiv \, \fbox{$ab$} \, ab^{2} \, \fbox{$ba$} \, .
\end{equation}
We apply Proposition~\ref{multieqlambdashuffle} for factors of multi-degree $(1, \, 1)$ on \eqref{nyxterino}. Factoring out the term
${\lambda}(ab)$ and working modulo $\ker \lambda$ we get $-2 \, b^{3}a^{2} \, \equiv \, ab^{3}a \, + \, b^{2}aba$. Factoring out the letter $a$
we get $-2 \, b^{3}a \, \equiv \, 2 \, ab^{3} \, + \, b^{2}ab$, which in turn yields $b^{2}ab \, \equiv \, 0$, a contradiction.

\subsubsection{The case $\mathbf{{|s|}_{a} \, \geq \, 1}.$}

We deal with the equation
\begin{equation} \label{innovator}
{\eta}_{1} \, a^{k}b^{2}s{\tilde{s}}b^{2}a^{k+1} \: \equiv \: {\eta}_{2} \, \fbox{$a^{k}ba^{l}$} \, btb^{2}a^{k} \, .
\end{equation}
Suppose that $l > 1$. Then applying Proposition~\ref{multieqlambdashuffle} for factors of multi-degree $(k+l, \, 1)$ on \eqref{innovator}
and working modulo $\ker \lambda$ we finally get $btb^{2}a^{k} \, \equiv \, 0$, a clear contradiction.

\smallskip

Suppose now that $l = 1$. From this point and on we check the validity of
\begin{equation} \label{halilai}
\fbox{$a^{k}b$} \, bs{\tilde{s}}b^{2}a^{k+1} \: \equiv \: \pm \; \fbox{$a^{k}b$} \, abtb \, \fbox{$ba^{k}$} \, .
\end{equation}
We apply Proposition~\ref{multieqlambdashuffle} for factors of multi-degree $(k, \, 1)$ on \eqref{halilai}. Factoring out the polynomial
${\lambda}(a^{k}b)$ and working modulo $\ker \lambda$ with Proposition~\ref{Ree} we obtain
\begin{equation} \label{arkalovolivianh}
bs{\tilde{s}}b^{2}a^{k+1} \: \equiv \: \pm \; \bigl( abtb^{2}a^{k} \, + \, b{\tilde{t}}baba^{k} \bigr) \, .
\end{equation}
From this we immediately obtain
\begin{equation} \label{anastasia}
s{\tilde{s}}b^{2}a^{k+1} \, \equiv \, \pm \, {\tilde{t}}baba^{k}.
\end{equation}
By \eqref{k+l} we also get 
\begin{equation} \label{jemaa}
bs{\tilde{s}}b^{2} \, \equiv \, \mp \, (k + 1) \, btb^{2}.
\end{equation}

$\mathbf{A. \, k}$ \textbf{even}. Both words in \eqref{anastasia} have odd length so our induction
hypothesis immediately discards the minus case. For the plus case it yields $s{\tilde{s}}b^{2}a^{k+1} \, = \, a^{k}babt$.
It is not then hard to check that $s = a^{k}babp$ and $t = p{\tilde{p}}baba^{k}b^{2}a^{k+1}$, for some $p \in A^{*}$.
Then \eqref{jemaa} becomes
$ba^{k}babp{\tilde{p}}baba^{k}b^{2} \, \sim \, bp{\tilde{p}}baba^{k}b^{2}a^{k+1}b^{2}$. By Corollary~\ref{ewdw} we get
$e_{b} \big( ba^{k}babp{\tilde{p}}baba^{k}b^{2} \bigr) = 2$, so we must also have $e_{b} \big( bp{\tilde{p}}baba^{k}b^{2}a^{k+1}b^{2} \bigr) = 2$.
This can only be achieved if $p{\tilde{p}}baba^{k}b^{2}a^{k+1} \, \equiv \, 0$.
It follows that $a^{k+1}b^{2}a^{k}babp{\tilde{p}} \, = \, p{\tilde{p}}baba^{k}b^{2}a^{k+1}$, hence by Lemma~\ref{1thL+S} the words
$a^{k+1}b^{2}a^{k}bab$ and $baba^{k}b^{2}a^{k+1}$ must be conjugate, which is easily checked to be impossible.

\smallskip

$\mathbf{B. \, k}$  \textbf{odd}. Suppose that $s{\tilde{s}}b^{2}a^{k+1}, {\tilde{t}}baba^{k} \, \not \equiv \,0$.
This time \eqref{anastasia} with the plus sign is immediately discarded by our induction hypothesis since
$s{\tilde{s}}b^{2}a^{k+1} = {\tilde{t}}baba^{k}$ clearly can not hold.
For the minus sign in \eqref{anastasia} we respectively obtain $s{\tilde{s}}b^{2}a^{k+1} \, = \, a^{k}babt$ and we argue entirely as in the case
where $k$ is even.

Finally let us consider the case where $s{\tilde{s}}b^{2}a^{k+1} \, \equiv \, 0 \, \equiv \,  {\tilde{t}}baba^{k}$.
Then $t \, = \, p{\tilde{p}}baba^{k}$, for some $p \in A^{*}$ and by Lemma~\ref{1thL+S} $s{\tilde{s}} \, = \, a^{k+1}{(b^{2}a^{k+1})}^{d}$,
for some $d \geq 0$, so that \eqref{jemaa} becomes
\begin{equation}
{\lambda}(ba^{k+1}{(b^{2}a^{k+1})}^{d}b^{2}) \: = \: \mp \; (k + 1) \, {\lambda}(bp{\tilde{p}}baba^{k}b^{2}) \, ,
\end{equation}
By Lemma~\ref{gamma1} ${\gamma} \bigl( b{(a^{k+1}b^{2})}^{d+1} \bigr) = 1$, so that we must have $(k + 1) \, | \, 1$,
a contradiction.

\medskip

The following result follows from our global analysis of \S 7.

\begin{lemma} \label{7for8}
Suppose that $a^{k}bx{\tilde{x}}ba^{k+1} \: \sim \: a^{k}byba^{k}$, with ${|x|}_{b} \geq 1$. Then $x = bs$ and $y = abp$, where $p = uba$ or
$p = vb$ for some $s, u, v \in A^{*}$.
\end{lemma}

\begin{proof}
The only cases where our hypotheses might hold are either when $x = bs$ and $y = abtba$ or $x = bs$ and $y = abtb$, studied in \S 7.4 and \S 7.5 respectively.
\end{proof}

\bigskip

%%%%%%%%%%%%%%%%%%%%%%%%%%%%%%%%%% a^{k}ba^{l}bx{\tilde{x}}ba^{l} = +- a^{k+1}byba^{l-1}, k+1 < l-1 %%%%%%%%%%%%%%%%%%%%%%%%%%%%%%%%%%%%%%%%%%%%%%%%%%
%%%%%%%%%%%%%%%%%%%%%%%%%%%%%%%%%%%%%%%%%%%%%%%%%%%%%%%%%%%%%%%%%%%%%%%%%%%%%%%%%%%%%%%%%%%%%%%%%%%%%%%%%%%%%%%%%%%%%%%%%%%%%%%%%%%%%%%%%%%%%%%%%%%%%%

\section{The case $\mathbf{a^{k}ba^{l}bx{\tilde{x}}ba^{l} \, \equiv \, \pm \, a^{k+1}byba^{l-1}, \:
k+1 < l-1 \, }.$}

\subsection{The case $\mathbf{{|x|}_{b} \, \geq \, 1}.$}

We necessarily have ${|y|}_{b} > 1$ and ${|y|}_{b}$ odd. By \eqref{k} we obtain
\begin{equation} \label{ybal-1}
yba^{l-1} \: \equiv \: \pm \, (k+1) \, a^{l}bx{\tilde{x}}ba^{l-1}.
\end{equation}
Suppose that $y = bt$, for some $t \in A^{+}$. Then we necessarily have $tba^{l-1} \, \equiv \, 0$, which
is contrary to the fact that ${|y|}_{b}$ is odd.

If $y = a^{m}bt$, for some $m \geq 1$ and $t \in A^{+}$ \eqref{ybal-1} becomes
\begin{equation} \label{ambtbal-1}
a^{m}btba^{l-1} \: \equiv \: \pm \, (k+1) \, a^{l}bx{\tilde{x}}ba^{l-1}.
\end{equation}

{$\mathbf{A. \; btb \, \not \equiv \, 0}$.}
By Corollary~\ref{ewdw} $e(a^{m}btba^{l-1}) = m + l - 1$ and $e(a^{l}bx{\tilde{x}}ba^{l-1}) = 2l - 2$.
Thus we get $m = l - 1$ and we have
\begin{equation}
\fbox{$ \, a^{l-1}bx{\tilde{x}}ba^{l} \: \sim \: a^{l-1}btba^{l-1} \, $} \, ,
\end{equation}
where ${|x|}_{b} \geq 1$. By Lemma~\ref{7for8} then we obtain $x = bs$ and $t = abp$, where either $p = uba$ or $p = vb$,
for some $u, v \in A^{*}$. Then we need to check if
\begin{equation}
\fbox{$a^{k}ba^{l}$} \, b^{2}s{\tilde{s}}b^{2}a^{l} \: \equiv \: \pm \;
\fbox{$a^{k+1}ba^{l-1}$} \, babpba^{l-1} \, .
\end{equation}
Applying Proposition~\ref{multieqlambdashuffle} for factors of multi-degree $(k+l, \, 1)$ we get
\begin{equation}
b^{2}s{\tilde{s}}b^{2}a^{l} \: \sim \: \fbox{$ba$} \, bpba^{l-1} \, .
\end{equation}
Applying Proposition~\ref{multieqlambdashuffle} for factors of multi-degree $(1, \, 1)$ and working modulo $\ker \lambda$ we finally obtain
$bpba^{l-1} \, \equiv \, 0$, which is a clear contradiction.

\medskip

{$\mathbf{B. \; btb \, \equiv \, 0}$.}
Going back to \eqref{ambtbal-1} we get $m = l$ by Corollary~\ref{ewdw}. Since ${|t|}_{b}$ is even we distinguish between the following two cases.

If $t = b^{2n}$, for some $n \geq 1$ then ${|x|}_{a} = 0$ and \eqref{ambtbal-1} becomes
$a^{l}b^{2n+2}a^{l-1} \: \equiv \: \pm \, (k+1) \, a^{l}b^{2n+2}a^{l-1}$, which in any case yields
$a^{l}b^{2n+2}a^{l-1} \, \equiv \, 0$, a contradiction.

Now suppose that $t = z{\tilde{z}}$ for some $z \in A^{+}$ with ${|z|}_{a} \geq 1$.
If $z = bu$, for some $u \in A^{+}$  we need to check if
\begin{equation}
a^{k}ba^{l}bx{\tilde{x}}ba^{l} \, \equiv \, \pm \, \underbrace{a^{k+1}ba^{l}b^{2}u{\tilde{u}}} \, b \, \fbox{$ba^{l-1}$} \, .
\end{equation}
We apply Proposition~\ref{multieqlambdashuffle} for all factors of multi-degree $(0, \, 1)$ and factoring out the letter $b$
we obtain
\begin{equation} \label{aggelina}
 \left \{ \begin{array}{c}
  + \, a^{k} \: \sh \:   a^{l}bx{\tilde{x}}ba^{l}  \\
  \pm \, a^{l}ba^{k} \: \sh \: x{\tilde{x}}ba^{l} \\
            \vdots                \\
  \mp \, x{\tilde{x}}ba^{l}ba^{k} \: \sh \: a^{l}
      \end{array} \right \} \: = \:  \pm \,
 \left \{ \begin{array}{c}
 - \, a^{k+1} \: \sh \: a^{l}b^{2}u{\tilde{u}}b^{2}a^{l-1} \\
 \mp \, a^{l}ba^{k+1} \: \sh \: bu{\tilde{u}}b^{2}a^{l-1}  \\
 \pm \, ba^{l}ba^{k+1} \: \sh \: u{\tilde{u}}b^{2}a^{l-1} \\
           \vdots   \\
 \mp \, u{\tilde{u}}b^{2}a^{l}ba^{k+1} \: \sh \: \fbox{$ba^{l-1}$}  \\
 \pm \, bu{\tilde{u}}b^{2}a^{l}ba^{k+1} \: \sh \: a^{l-1}
    \end{array} \right  \} .
\end{equation}
Then we act by $[a^{l-1}b] \, \rhd$ on \eqref{aggelina} using Proposition~\ref{minhpetitot}. Then working modulo $\ker \lambda$ we finally
obtain $u{\tilde{u}}b^{2}a^{l}ba^{k+1} \, \equiv \, 0$, which is a contradiction since the number of occurrences
of the letter $b$ in it is odd.

If $z = a^{n}bu$, for some $n \geq 1$ and $u \in A^{*}$ we need to check whether
\begin{equation}
a^{k}ba^{l}bx{\tilde{x}}ba^{l} \, \equiv \, \pm \, a^{k+1}ba^{l}ba^{n}bu{\tilde{u}}ba^{n}ba^{l-1} \, .
\end{equation}
It follows that $a^{l}bx{\tilde{x}} \, \sim \, a^{l}ba^{n}bu{\tilde{u}}ba^{n}$, hence we get $x = a^{n}bs$, for some
$s \in A^{*}$ so that we actually compare
\begin{equation} \label{ella}
a^{k}ba^{l}ba^{n}bs{\tilde{s}}b \, \fbox{$a^{n}ba^{l}$} \, \equiv \, \pm \,
\fbox{$a^{k+1}ba^{l}$} \, ba^{n}bu{\tilde{u}}ba^{n}ba^{l-1}.
\end{equation}
We distinguish between three cases.

(1) $n < k + 1$. Then $k + l + 1$ is strictly larger than $n + l, k + l$ and $n + l - 1$. We then apply Proposition~\ref{multieqlambdashuffle}
on \eqref{ella} for all factors of multi-degree $(k+l+1, \, 1)$ and we finally get $ba^{n}bu{\tilde{u}}ba^{n}ba^{l-1} \, \equiv \, 0$,
which is again impossible.

(2) $n > k + 1$. Then $n + l$ is strictly larger than $k + l, k + l + 1$ and $n + l - 1$. In this case we apply Proposition~\ref{multieqlambdashuffle}
for all factors of multi-degree $(n+l, \, 1)$ on \eqref{ella} and get $bs{\tilde{s}}ba^{n}ba^{l}ba^{k} \, \equiv \, 0$, another contradiction.

(3) $n = k + 1$. Proposition~\ref{multieqlambdashuffle} for factors of multi-degree $(n+l, \, 1)$ will yield
\begin{equation}
bs{\tilde{s}}ba^{k+1}ba^{l}ba^{k}  \: \equiv \: \, \pm  \, ba^{k+1}bu{\tilde{u}}ba^{k+1}ba^{l-1}.
\end{equation}
Since both words clearly lie in the support of the free Lie algebra our induction hypothesis implies that they are either equal or one is 
the reversal of the other. The former can not hold since $k < l - 1$ and the latter is clearly impossible.  

\bigskip

\subsection{The case $\mathbf{{|x|}_{b} \, = \, 0}.$}

We have to consider the equation
\begin{equation}
a^{k}ba^{l}ba^{2m}ba^{l} \: \equiv \: \pm \, a^{k+1}ba^{p}ba^{q}ba^{l-1} \, ,
\end{equation}
where $m,p,q \geq 0$ and $p + q = l + 2m$. By \eqref{k+1} and \eqref{k} we have
\begin{equation} \label{821}
(k + 1) \, a^{l}ba^{2m}ba^{l-1} \, \equiv \, \mp \,  a^{p}ba^{q}ba^{l-1}.
\end{equation}
By Corollary~\ref{ewdw} it follows that $e(a^{l}ba^{2m}ba^{l-1}) = 2l -2$ and $e(a^{p}ba^{q}ba^{l-1})$ is either equal to $p + l - 1$,
when $q$ is odd; or to $p + l - 2$, if $q$ is even. We respectively obtain $p = l - 1$ or $p = l$. The latter clearly can not hold since
\eqref{821} would yield $a^{l}ba^{2m}ba^{l-1} \, \equiv \, 0$. Therefore we are left with the case $p = l - 1$, i.e.,
\begin{equation} \label{822}
a^{k}ba^{l}ba^{2m}ba^{l} \; \equiv \; \pm \, a^{k+1}ba^{l-1}ba^{2m+1}ba^{l-1}, \quad k + 1 < l - 1.
\end{equation}
By \eqref{k+l} it follows that ${(-1)}^{k} \, \binom{k+l}{k} \, ba^{l}ba^{2m}b \, \equiv \, \pm \,
{(-1)}^{k+1} \, \binom{k+l}{k+1} \, ba^{l-1}ba^{2m+1}b$, which in turn yields
${(-1)}^{k} \, \binom{k+l}{k} \, {(-1)}^{1} \, \binom{1+1}{1} \, a^{l}ba^{2m} \, \equiv \, \pm \,
{(-1)}^{k+1} \, \binom{k+l}{k+1} \, {(-1)}^{1} \, \binom{1+1}{1} \, a^{l-1}ba^{2m+1}$. By Lemma~\eqref{lambda1b} then we finally obtain
\begin{equation} \label{823}
\binom{k+l}{k} \binom{l+2m}{l} \; = \; \pm \, \binom{k+l}{k+1} \binom{l+2m}{l-1}.
\end{equation}
The minus sign in \eqref{823} - which respectively corresponds to the minus in \eqref{822} and the plus in \eqref{821} - is clearly impossible. Calculating the binomial coefficients we are finally left with the equation
\begin{equation} \label{824}
(k + 1)(2m + 1) \; = \; l^{2},
\end{equation}
which is feasible under our assumptions, e.g., take $k = 3, l = 6$ and $m = 4$.

Equation \eqref{821} with the minus sign and for $p = l - 1$ becomes
\begin{equation}
(k + 1) \, a^{l}ba^{2m}ba^{l-1} \, \equiv \, - \,  a^{l-1}ba^{2m+1}ba^{l-1}.
\end{equation}
By \eqref{k} we obtain
${(-1)}^{l} \, (k+1) \, a^{2m}ba^{l} \; = \; - \, 2 \, {(-1)}^{l} \, a^{2m+1}ba^{l-1}$.
By Lemma~\ref{lambda1b} then it follows that
\begin{equation}
(k+1) \, \binom{l+2m}{l} \; = \; 2 \, \binom{l+2m}{l-1},
\end{equation}
which finally yields
\begin{equation} \label{825}
(k + 1)(2m + 1)\; = \; 2 \, l.
\end{equation}
From \eqref{824} and \eqref{825} we get $l^{2} = 2l$, i.e., $l = 2$, which contradicts our assumption that $k+1 < l-1$.

\bigskip

%%%%%%%%%%%%%%%%%%%%%%%%%%%%%%%%%%%%%%%%%%%%%%%%%%%%%%%%%%%%%%%%%%%%%%%%%%%%%%%%%%%%%%%%%%%%%%%%%%%%%%%%%%%%%%%%%%%%%%%%%%%%%%%%%%%%%%%%%%%%%%%%%%%%%%%%
%%%%%%%%%%%%%%%%%%%%%%%%%%%%%%%%%%%%%%%%%%%%%%%%%%%%%%%%%%%%%%%%%%%%%%%%%%%%%%%%%%%%%%%%%%%%%%%%%%%%%%%%%%%%%%%%%%%%%%%%%%%%%%%%%%%%%%%%%%%%%%%%%%%%%%%%

\section{The case $\mathbf{a^{k}ba^{k+2}bx{\tilde{x}}ba^{k+2} \, \equiv \, \pm \, a^{k+1}byba^{k+1}}.$}

\subsection{The case $\mathbf{{|x|}_{b} \, = \, 0}.$}

Then ${|y|}_{b} = 1$, i.e., we consider the comparison
\begin{equation} \label{901}
a^{k}ba^{k+2}ba^{2l}ba^{k+2} \: \equiv \: \pm \; a^{k+1}ba^{p}ba^{q}ba^{k+1} \, ,
\end{equation}
where $l,p \geq 0$ and $p + q = k + 2l + 2$.
Applying Proposition~\ref{multieqlambdashuffle} for factors of multi-degree $(0, \, 1)$ on \eqref{901} we factor out the letter $b$ and we obtain
\begin{equation} \label{902}
 \left \{ \begin{array}{c}
  + \, {(-1)}^{k} \, a^{k} \: \sh \:   a^{k+2}ba^{2l}ba^{k+2}  \\
 - \, a^{k+2}ba^{k} \: \sh \: a^{2l}ba^{k+2} \\
 + \,  a^{2l}ba^{k+2}ba^{k} \: \sh \: a^{k+2}
      \end{array} \right \} \, = \, \pm \,
 \left \{ \begin{array}{c}
 + \, {(-1)}^{k+1} \, a^{k+1} \: \sh \: a^{p}ba^{q}ba^{k+1} \\
 + \, {(-1)}^{k+p} \, \fbox{$a^{p}ba^{k+1}$} \: \sh \: a^{q}ba^{k+1}  \\
 - \, a^{q}ba^{p}ba^{k+1} \: \sh \: a^{k+1}
          \end{array} \right  \} .
\end{equation}
Suppose that $p \neq k + 1$ and $p \neq 2l + 1$. Then acting by $[a^{p+k+1}b] \, \rhd$ on \eqref{902} we finally get $a^{q}ba^{k+1} \, \equiv \, 0$,
a contradiction. Therefore we may assume, without loss of generality, that $p = k + 1$, i.e.,
\begin{equation}
a^{k}ba^{k+2}ba^{2l}ba^{k+2} \: \equiv \: \pm \; a^{k+1}ba^{k+1}ba^{2l+1}ba^{k+1} \, .
\end{equation}
By successive applications of \eqref{k+l} we obtain
\begin{equation}
 \binom{2k+2}{k} \, \binom{k+2l+2}{k+2} \: = \: \pm \; \binom{2k+2}{k+1} \, \binom{k+2l+2}{k+1}.
\end{equation}
This is clearly impossible for the minus sign. For the plus one it is equivalent to having ${(k+2)}^{2} \: = \: (k+1) \, (2l+1)$.
But this is also a contradiction because if $k$ is even (resp. odd) the left hand side is also even (resp. odd) but the right hand side
is odd (resp. even).

\medskip

\subsection{The case $\mathbf{{|x|}_{b} \, \geq \, 1}.$}

Then ${|y|}_{b} > 1$ and ${|y|}_{b}$ is odd. By \eqref{k} and \eqref{k+1} we obtain
\begin{equation} \label{903}
{\lambda}(a^{k+1}by) \: + \: {(-1)}^{k} \, {\lambda}(yba^{k+1}) \: = \:
\pm \, (k + 1) \, {\lambda}(a^{k+2}bx{\tilde{x}}ba^{k+1}).
\end{equation}
If $y = aub$ for some $u \in A^{*}$, it follows that $a^{k+1}bau \, \equiv \, 0$, which contradicts the assumption
that ${|y|}_{b}$ is odd. Similarly we deal with the case where $y = bua$.

Suppose that $y = bub$, for some $u \in A^{+}$ and let $P$ and $Q$ denote respectively the left and the right hand side in \eqref{903}.
Then $e(P) \leq k + 1$, whereas $e(Q) = 2k + 2$ and we reach a contradiction.

\smallskip

In the sequel we suppose that $y$ starts and ends with the letter $a$, i.e., $y = a^{l}btba^{m}$ for some $t \in A^{*}$ and $l,m \geq 1$,
where without loss of generality $l \leq m$.

Suppose that $x$ starts with $b$, i.e., $x = b^{n-1}z$ for $n \geq 2$ and $z \in A^{*}$. Then we have
\begin{equation} \label{bladi}
a^{k}ba^{k+2}b^{n}z{\tilde{z}}b^{n}a^{k+2} \, \equiv \, \pm \, a^{k+1}ba^{l}btba^{m}ba^{k+1}.
\end{equation}
Observe that $ba^{k+2}b^{n}z{\tilde{z}}b^{n} \, \sim \, ba^{l}btba^{m}b$. Since ${|t|}_{b}$ is odd
$a^{l}btba^{m} \, \not \equiv \, 0$, therefore by Corollary~\ref{ewdw} we get $e_{b}(ba^{l}btba^{m}b) \, = \, 2$.
On the other hand, $e_{b} \bigl( ba^{k+2}b^{n}z{\tilde{z}}b^{n} \bigr) \geq n$; the equality holds in the case where 
$a^{k+2}b^{n}z{\tilde{z}} \, \equiv \, 0$.
It follows that $n = 2$, $k$ is even and $a^{k+2}b^{2}z{\tilde{z}} = z{\tilde{z}}b^{2}a^{k+2}$. By Lemma~\ref{1thL+S} we then obtain
$z{\tilde{z}} = a^{k+2}{(b^{2}a^{k+2})}^{d}$, for some $d \geq 0$, so that \eqref{bladi} becomes
\begin{equation}
a^{k}ba^{k+2}{(b^{2}a^{k+2})}^{d+1}b^{2}a^{k+2} \, \equiv \, \pm \, a^{k+1}ba^{l}btba^{m}ba^{k+1}.
\end{equation}
By \eqref{k+l} we get
\begin{equation}
\binom{2k + 2}{k} \, ba^{k+2}{(b^{2}a^{k+2})}^{d+1}b^{2} \, \equiv \, \mp \, \binom{2k + 2}{k + 1} \, ba^{l}btba^{m}b,
\end{equation}
It follows that
\begin{equation}
(k + 1) \, {\gamma} \bigl( b{(a^{k+2}b^{2})}^{d+2} \bigr) \; = \; (k + 2) \, {\gamma}(ba^{l}btba^{m}b).
\end{equation}
Since $k$ is even Lemma~\ref{gamma1} yields ${\gamma} \bigl( b{(a^{k+2}b^{2})}^{d+2} \bigr) = 1$. But then $(k + 2) \, | \, (k + 1)$,
a contradiction.

\bigskip

We are therefore left with the case where $x = a^{n}bs$, for some $n \geq 1$, i.e.,
\begin{equation} \label{904}
\fbox{$ \, a^{k}ba^{k+2}ba^{n}bs{\tilde{s}}ba^{n}ba^{k+2} \, \equiv \, \pm \, a^{k+1}ba^{l}btba^{m}ba^{k+1} \, $} \, ,
\end{equation}
where $l \leq m$, without loss of generality. Since ${|t|}_{b}$ is odd $a^{m}btba^{n} \, \not \equiv \, 0$. Then \eqref{904} implies that
\begin{equation} \label{kalbi}
a^{k+2}ba^{n}bs{\tilde{s}}ba^{n} \, \sim \, a^{l}btba^{m} \, .
\end{equation}
Clearly $e \bigl( a^{k+2}ba^{n}bs{\tilde{s}}ba^{n} \bigr) = k + n + 2$. We have to distinguish between two cases.

\smallskip

{$\mathbf{A. \; {|t|}_{a} \, = \, 0}$.}
Then $e(a^{l}btba^{m}) = l + m -1$, so that $k + n + 2 = l + m - 1$. On the other hand,
counting occurrences of the letter $a$ in \eqref{kalbi} we get $k + 2n + 2 + 2{|s|}_{a} = l + m$.
It follows that $n + 2{|s|}_{a} = 1$ which clearly implies that $n = 1$ and ${|s|}_{a} = 0$. Then \eqref{kalbi} reads
\begin{equation} \label{eladia}
\fbox{$ab$} \, b^{2q-1}aba^{k+2} \, \sim \, a^{l}b^{2q + 1}a^{m} \, ,
\end{equation}
for some $q \geq 1$ (for $q = 0$ we fall back to the case where ${|x|}_{b} = 0$).
If $l > 1$, then applying Proposition~\ref{multieqlambdashuffle} for factors of multi-degree $(1, \, 1)$ we finally get
$b^{2q-1}aba^{k+2} \, \equiv \, 0$, a contradiction. Thus we must have $l = 1$, i.e.,
\begin{equation}
\fbox{$ab$} \, b^{2q-1}aba^{k+2} \, \sim \, \fbox{$ab$} \, b^{2q}a^{k+3} \, .
\end{equation}
Factoring out ${\lambda}(ab)$ once more by Proposition~\ref{multieqlambdashuffle} we get $a^{k+2}bab^{2q-1} \, \sim \, a^{k+3}b^{2q}$.
But the latter is a contradiction since $a^{k+2}bab^{2q-1}$ and $a^{k+3}b^{2q}$ are distinct Lyndon words and due to
Proposition~\ref{lambdaLyndon} the polynomials ${\lambda}(a^{k+2}bab^{2q-1})$ and ${\lambda}(a^{k+3}b^{2q})$ are linearly independent in
$K \langle A \rangle$.

\smallskip

{$\mathbf{B. \; {|t|}_{a} \, \geq \, 1}$.}
In that case \eqref{kalbi} yields
\begin{equation} \label{905}
k + n + 2 \, = \, l + m.
\end{equation}
Equation \eqref{903} then becomes
\begin{equation} \label{906}
{\lambda}(a^{k+1}ba^{l}btba^{m}) \: + \: {(-1)}^{k} \, {\lambda}(a^{l}btba^{m}ba^{k+1}) \: = \:
\pm \, (k + 1) \, {\lambda}(a^{k+2}ba^{n}bs{\tilde{s}}ba^{n}ba^{k+1}).
\end{equation}
Let $P$ and $Q$ denote respectively the left and the right hand side of \eqref{906}. By Corollary~\ref{ewdw} it is clear that
$e(Q) = 2k + 2$ and $e(P) \leq k + m + 1$. It follows that $2k + 2 \leq k + m + 1$, i.e., $k + 1 \leq m$.

\begin{lemma} \label{klm}
With $k, l, m, n$ as above we claim that $l = m = k+1$ and $n = k$.
\end{lemma}

\begin{proof}
First we show that $l = k + 1$. Suppose that $k + 1 < l$. Using \eqref{905} it is easy then to check that $k + n + 2$ is strictly larger than
$2k + 2$, $k + l + 1$ and $k + m + 1$. This permits us to apply Proposition~\ref{multieqlambdashuffle} on \eqref{904} for factors of multi-degree
$(k + n + 2, \, 1)$, and finally obtain the contradiction $bs{\tilde{s}}ba^{n}ba^{k+2}ba^{k} \, \equiv \, 0$. Therefore we necessarily have
\begin{equation} \label{907}
l \, \leq \, k + 1 \, \leq \, m.
\end{equation}
Now suppose that $l < k + 1$. Then clearly $l < m$ and since ${|t|}_{b}$ is odd we get
$tba^{m}ba^{k+1} \, \not \equiv \, 0$ and therefore $d(P) = l$ in \eqref{906}. On the other hand, $d(Q) = k + 1$, so we obtain $k + 1 = l$, a contradiction.

Now we are ready to show that $m = k + 1$. Since $l = k + 1$ \eqref{905} yields $m = n + 1$, so we must show that $k = n$.
Equation \eqref{906} then reads
\begin{equation} \label{908}
{\lambda}(a^{k+1}ba^{k+1}btba^{n+1}) \: + \: {(-1)}^{k} \, {\lambda}(a^{k+1}btba^{n+1}ba^{k+1}) \: = \:
\pm \, (k + 1) \, {\lambda}(a^{k+2}ba^{n}bs{\tilde{s}}ba^{n}ba^{k+1}).
\end{equation}
Suppose, for the sake of contradiction, that $k < n$ - it is enough to discard this case in view of \eqref{907}.
Then we would necessarily have $a^{k+1}bt \, \equiv \, 0$. Indeed,
if not, Corollary~\ref{ewdw} would yield $e(P) = k + n + 2$. On the other hand $2k + 2 = e(Q)$, so we would get $k = n$, a contradiction.
Thus we must have $t = p{\tilde{p}}ba^{k+1}$ for some $p \in A^{*}$ and $e \bigl( a^{k+1}ba^{k+1}bp{\tilde{p}}ba^{k+1}ba^{n+1} \bigr) = k + n + 1$.

If $n > k + 1$ then since $e \bigl( a^{k+1}bp{\tilde{p}}ba^{k+1}ba^{n+1}ba^{k+1} \bigr) \leq 2k + 2$, we obtain $e(P) = k + n + 1$. This must also
be equal to $e(Q) = 2k + 2$, so that $n = k + 1$, a contradiction.

Therefore we may assume that $n = k + 1$, so that \eqref{904} becomes
\begin{equation} \label{909}
a^{k}ba^{k+2}ba^{k+1}bs{\tilde{s}}b \, \fbox{$a^{k+1}ba^{k+2}$} \, \equiv \:
\pm \, a^{k+1}ba^{k+1}bp{\tilde{p}}ba^{k+1}b \, \fbox{$a^{k+2}ba^{k+1}$} \, .
\end{equation}
We apply Proposition~\ref{multieqlambdashuffle} for factors of multi-degree $(2k + 3, \, 1)$ on \eqref{909} and finally obtain
\begin{equation} \label{910}
bs{\tilde{s}}ba^{k+1}ba^{k+2}ba^{k} \, \equiv \, \mp \, ba^{k+1}bp{\tilde{p}}ba^{k+1}ba^{k+1},
\end{equation}
since $a^{k+1}ba^{k+2} \, \equiv \, - \, a^{k+2}ba^{k+1}$.
Clearly both words in \eqref{910} lie in the support of the free Lie algebra. Then our induction hypothesis yields an immediate contradiction
or it implies either
$bs{\tilde{s}}ba^{k+1}ba^{k+2}ba^{k} = ba^{k+1}bp{\tilde{p}}ba^{k+1}ba^{k+1}$, or
$bs{\tilde{s}}ba^{k+1}ba^{k+2}ba^{k} = a^{k+1}ba^{k+1}p{\tilde{p}}ba^{k+1}b$. It is evident that both equalities can not hold.
\end{proof}

\medskip

In view of Lemma~\ref{klm} equation \eqref{904} finally becomes
\begin{equation} \label{911}
a^{k}ba^{k+2}ba^{k}bs{\tilde{s}}ba^{k}ba^{k+2} \: \equiv \: \pm \, a^{k+1}ba^{k+1}btba^{k+1}ba^{k+1}.
\end{equation}
We apply Proposition~\ref{multieqlambdashuffle} for factors of multi-degree $(0, \, 1)$ on \eqref{911} and obtain
\begin{equation} \label{912}
\!\!\!\! \left \{ \begin{array}{c} \!\!\!
  +  a^{k} \: \sh \:   {\fbox{$a^{k+2}ba^{k}bs{\tilde{s}}ba^{k}$}} \, ba^{k+2}  \\
  +  {(-1)}^{k+1} \, a^{k+2}ba^{k} \: \sh \: {\fbox{$a^{k}bs{\tilde{s}}ba^{k}ba^{k+2}$}}  \\
               \vdots                \\
  -  {\fbox{$s{\tilde{s}}ba^{k}ba^{k+2}ba^{k}$}} \: \sh \: a^{k}ba^{k+2} \\
  +  {(-1)}^{k} \, {\fbox{$a^{k}bs{\tilde{s}}ba^{k}ba^{k+2}$}} \, ba^{k} \: \sh \:  a^{k+2}
      \end{array} \right \} = \pm
 \left \{ \begin{array}{c} \!\!\!
 -  a^{k+1} \: \sh \: {\fbox{$a^{k+1}btba^{k+1}$}} \, ba^{k+1} \\
 +  {(-1)}^{k+1} \, a^{k+1}ba^{k+1} \: \sh \: {\fbox{$tba^{k+1}ba^{k+1}$}} \\
           \vdots   \\
 -  {\fbox{${\tilde{t}}ba^{k+1}ba^{k+1}$}}  \: \sh \: a^{k+1}ba^{k+1} \\
 +  {(-1)}^{k+1} \, {\fbox{$a^{k+1}b{\tilde{t}}ba^{k+1}$}} \, ba^{k+1}  \: \sh \: a^{k+1}
    \end{array} \right  \} .
\end{equation}
For the moment, let $P$ be an arbitrary fixed Lie polynomial.
Consider the integer coefficients of the words in $P$ that are marked in \eqref{912} by boxes. More precisely let
${\eta}_{1} \, = \, (P, \, a^{k}bs{\tilde{s}}ba^{k}ba^{k+2})$,
$p \, = \, (P, \, s{\tilde{s}}ba^{k}ba^{k+2}ba^{k})$,
${\eta}_{2} \, = \, (P, \, a^{k+1}btba^{k+1})$, $q \, = \, (P, \, tba^{k+1}ba^{k+1})$ and
$r \, = \, (P, \, {\tilde{t}}ba^{k+1}ba^{k+1})$.
When $k$ is even (resp. odd) the common length of these words is odd (resp. even), therefore we get
$(P, a^{k+2}ba^{k}bs{\tilde{s}}ba^{k}) \, = \, {(-1)}^{k} \, {\eta}_{1}$ and
$(P, a^{k+1}b{\tilde{t}}ba^{k+1}) \, = \, {(-1)}^{k} \, {\eta}_{2}$.
If we act by $P \, \rhd$ on \eqref{912} then by Proposition~\ref{minhpetitot} we obtain
\begin{equation} \label{913}
 \left \{ \begin{array}{c}
  {(-1)}^{k}{\eta}_{1}\bigl[ a^{k} \sh  ba^{k+2} \, + \, a^{k+2}  \sh  ba^{k} \bigr] \\
  + \, {(-1)}^{k+1}{\eta}_{1}(a^{k+2}ba^{k}) \, - \, p(a^{k}ba^{k+2}) \\
          \end{array} \right \}  =  \pm
 \left \{ \begin{array}{c}
 - 2{\eta}_{2} (a^{k+1}  \sh  ba^{k+1}) \, + \\
 {(-1)}^{k+1}q(a^{k+1}ba^{k+1}) \, - \, r(a^{k+1}ba^{k+1})
    \end{array} \right  \} .
\end{equation}
Using Lemma~\ref{abshuffle} we equate the coefficients of $ba^{2k+2}$ in \eqref{913} and obtain the equation
\begin{equation} \label{914}
2 {(-1)}^{k} \, {\eta}_{1} \, \binom{2k+2}{k} \: = \: \mp \, 2 {\eta}_{2} \, \binom{2k+2}{k+1}.
\end{equation}
Since $a^{k}bs{\tilde{s}}ba^{k}ba^{k+2} \, \not \equiv \, 0$ there exists a Lie polynomial $P$ such that
${\eta}_{1} \, \neq 0$. Consider the coefficient $(P, a^{k+1}btba^{k+1})$ of this particular polynomial.
If $(P, a^{k+1}btba^{k+1}) = 0$, i.e., ${\eta}_{2} = 0$, then \eqref{914} yields a contradiction. Thus we may assume
that ${\eta}_{2} \neq 0$.

Suppose that $k \geq 3$.
Collecting all coefficients of $a^{2}ba^{2k}$ in \eqref{913} by Lemma~\ref{abshuffle} we obtain
\begin{equation} \label{915}
{(-1)}^{k} \, {\eta}_{1} \, \Big\{ \binom{2k}{k-2} \, + \, \binom{2k}{k} \Big\} \: = \: \mp \, 2{\eta}_{2} \, \binom{2k}{k-1}.
\end{equation}
Having assured that ${\eta}_{1}, {\eta}_{2} \, \neq 0$ we may divide \eqref{914} and \eqref{915} by parts and we obtain
\begin{equation} \label{916}
\frac{\displaystyle \binom{2k+2}{k}}{ \displaystyle \binom{2k}{k-2} \, + \, \binom{2k}{k}} \: = \:
\frac{\displaystyle \binom{2k+2}{k+1}}{ \displaystyle 2 \, \binom{2k}{k-1}}.
\end{equation}
It is easy to check, using the factorial definition of binomial coefficients, that \eqref{916} yields the absurd
$2k(k+1) \, = \, (k-1)k + (k+1)(k+2)$.

\smallskip

Suppose that $k = 1$. Our objective is to show that $s{\tilde{s}}baba^{3}ba \, \equiv \, 0$, which clearly
leads to a contradiction. To do this we need to show that $(P, \, s{\tilde{s}}baba^{3}ba) \, = \, 0$, for any Lie polynomial $P$,
which in this case is equivalent to showing that $p \, = \, 0$. By \eqref{914} we get
\begin{equation} \label{917}
- \, 4 \, {\eta}_{1} \, = \, \mp \, 6 \, {\eta}_{2}.
\end{equation}
Going back to \eqref{913} and collecting the coefficients of $aba^{3}$ by Lemma~\ref{abshuffle} we get
\begin{equation} \label{918}
- \, 4 \, {\eta}_{1} \, - \, p \, = \, \mp \, 6 \, {\eta}_{2}.
\end{equation}
By \eqref{917} and \eqref{918} it follows that $p \, = \, 0$, as required.

\smallskip

Let us finally deal with the case $k = 2$. Using Lemma~\ref{abshuffle} we collect all coefficients of the term $a^{2}ba^{4}$ from the shuffle
products $a^{2} \, \sh \, ba^{4}$,  $a^{3} \, \sh \, ba^{3}$  and $a^{4} \, \sh \, ba^{2}$ and we get
\begin{equation} \label{919}
7 \, {\eta}_{1} \, - \, p \: = \: \mp \, 8 \, {\eta}_{2}.
\end{equation}
On the other hand, by \eqref{915} we have $15 \, {\eta}_{1} \, = \, \mp \, 20 \, {\eta}_{2}$. Then from \eqref{919} it follows that
$p \, = \, {\eta}_{1}$. This implies that for each Lie polynomial $P$ we have
$(P, \, s{\tilde{s}}ba^{2}ba^{4}ba^{2}) \, = \, (P, \, a^{4}ba^{2}bs{\tilde{s}}ba^{2})$, i.e.,
\begin{equation}
s{\tilde{s}}ba^{2}ba^{4}ba^{2} \, \equiv \, a^{4}ba^{2}bs{\tilde{s}}ba^{2}.
\end{equation}
Since both words have odd length our induction hypothesis implies that either
$s{\tilde{s}}ba^{2}ba^{4}ba^{2} \, = \, a^{2}bs{\tilde{s}}ba^{2}ba^{4}$, which clearly can not hold, or
$s{\tilde{s}}ba^{2}ba^{4}ba^{2} \, = \, a^{4}ba^{2}bs{\tilde{s}}ba^{2}$. The latter yields
$a^{4}ba^{2}bs{\tilde{s}} \, = \, s{\tilde{s}}ba^{2}ba^{4}$. By Lemma~\ref{1thL+S} then there exist $u, v \in A^{*}$ such that
$a^{4}ba^{2}b \, = \, uv$, $ba^{2}ba^{4} \, = \, vu$ and $s{\tilde{s}} \, = \, u{(ba^{2}ba^{4})}^{n}$, for
some $n \geq 0$. Since $v$ must start and finish with the letter $b$ we either have $u = a^{4}$ and $v = ba^{2}b$,
or $u = a^{4}ba^{2}$ and $v = b$. The latter implies that $s{\tilde{s}}$ is a palindrome of odd length, a contradiction.
The former yields $s{\tilde{s}} \, = \, a^{4}{(ba^{2}ba^{4})}^{d}, \: d \geq 0$, so that \eqref{911} becomes
\begin{equation}
a^{2}ba^{4} \, \fbox{${(ba^{2}ba^{4})}^{d}b$} \, a^{2}ba^{4} \: \equiv \: \pm \, a^{3}ba^{3} \, \fbox{$btb$} \, a^{3}ba^{3}, \quad n \geq  1.
\end{equation}
By successive applications of \eqref{k+l} we obtain
\begin{equation}
{\binom{6}{2}}^{\!2} {(ba^{2}ba^{4})}^{d}b \; \equiv \; \pm \, {\binom{6}{3}}^{\!2} btb \, .
\end{equation}
It follows that
\begin{equation}
{\binom{6}{2}}^{\!2} {\gamma}({(ba^{2}ba^{4})}^{d}b) \; = \; {\binom{6}{3}}^{\!2} {\gamma}(btb) \, ,
\end{equation}
which yields $9 \, {\gamma}({(ba^{2}ba^{4})}^{d}b) \; = \; 16 \, {\gamma}(btb)$. In view of Lemma~\ref{gammaodd} the left hand side
is equal to an odd positive integer, so $16$ can not divide it and we reach a contradiction.

\bigskip

%%%%%%%%%%%%%%%%%%%%%%%%%%%%%%%%%%%%%%%%%%%%%%%%%%%%%%%%%%%%% FINAL TOUCH of the PROOF %%%%%%%%%%%%%%%%%%%%%%%%%%%%%%%%%%%%%%%%%%%%%%%%%%%%%%%%%%%%%%%
%%%%%%%%%%%%%%%%%%%%%%%%%%%%%%%%%%%%%%%%%%%%%%%%%%%%%%%%%%%%%%%%%%%%%%%%%%%%%%%%%%%%%%%%%%%%%%%%%%%%%%%%%%%%%%%%%%%%%%%%%%%%%%%%%%%%%%%%%%%%%%%%%%%%%%%%

\section{Proof of Theorem~\ref{theo:2}.}

Let $w_{1} \, \equiv \pm \, w_{2}$. In view of our analysis in \S 4 and \S 6 we may suppose that
$w_{1} = a^{k}buba^{l}$ and $w_{2} = a^{m}bvba^{n}$, where, without loss of generality, we have $k \leq l$ and $m \leq n$.

\smallskip

$\mathbf{A. \, bub, \, bvb \, \not \equiv \, 0}$. We may assume that $k \leq m$, otherwise we interchange $w_{1}$ and $w_{2}$.
By Corollary~\ref{ewdw} we have $e(w_{1}) = k + l$ and $e(w_{2}) = m + n$, hence $k + l = m + n$.
If $k = m$ this implies that $l = n$, and our result follows from \S 5.

If $k < m$ we let $m = k + r$ and $n = l - r$ for some positive integer $r$.
By Corollary~\ref{ewdw} we have $d(w_{1}) = k$, if $uba^{l} \, \not \equiv \, 0$ and $d(w_{1}) = k + 1$, otherwise.
On the other hand, $d(w_{2}) \geq k + r$. It follows that $r = 1$ and $uba^{l} \, \equiv \, 0$.
If $k + 1 < l - 1$ our result follows from \S 8, whereas if $k + 1 = l - 1$ it follows from \S 9.

\smallskip

$\mathbf{B. \, bub \, \equiv \, 0}$ \textbf{and} $\mathbf{bvb \, \not \equiv \, 0}$.
By Corollary~\ref{ewdw} we have $e(w_{1}) = k + l - 1$ and $e(w_{2}) = m + n$, and therefore $k + l - 1 = m + n$.

(1) If $u = x{\tilde{x}}$ for some $x \in A^{*}$ we must have $k < l$.
Since $x{\tilde{x}}ba^{l} \, \not \equiv \, 0$ we get $d(w_{1}) = k$. Our first claim is that $m = k$.
If $m < n$ we have $vba^{n} \, \not \equiv \, 0$ since ${|v|}_{b}$ is even,
so we get $d(w_{2}) = m$ and thus $m = k$. If $m = n$ Corollary~\ref{ewdw} implies that $d(w_{2}) \geq m$. It follows that $k \geq m$.
On the other hand, we have $2m = k + l - 1 \geq 2k$, since $l - 1 \geq k$, which implies that $m \geq k$. Therefore, in any case, we have $k = m$.
This implies that $n = l - 1$. If $k = l - 1$ our result follows from the analysis in \S 7.
If $k < l - 1$ equations \eqref{k} and \eqref{k+1} respectively yield $x{\tilde{x}}ba^{l} \, \equiv \, \pm \, vba^{l - 1}$ and
$x{\tilde{x}}ba^{l - 1} \, \equiv \, \pm \, vba^{l - 2}$. But these can not hold simultaneously.
Indeed, suppose that $x{\tilde{x}}ba^{l} \, \equiv \, vba^{l - 1}$. By our induction hypothesis, since both words are not even palindromes,
their common length must be odd, hence $l$ must be even. But then the equality $x{\tilde{x}}ba^{l - 1} \, \equiv \, vba^{l - 2}$ is
impossible because the corresponding common length is even. For the minus sign we use similar arguments.

(2) Suppose that $bub = u^{2d + 3}$, for some $d \geq 0$. Since $k + l - 1 = m + n$ we get ${|v|}_{a} = 1$.
Suppose that $k < l$ and $m < n$. If $vba^{n} \, \not \equiv \, 0$ Corollary~\ref{ewdw} yields $d(w_{1}) = k$ and $d(w_{2}) = m$, therefore
we get $k = m$, so that $n = l - 1$. It also implies that $b^{2d + 2}a^{l} \, \equiv \, \pm \, vba^{l - 1}$, which is impossible due to
our induction hypothesis. If $vba^{n} \, \equiv \, 0$ then we have $v = a^{n}bs{\tilde{s}}$, for some $s \in A^{*}$ and since ${|v|}_{a} = 1$
we obtain $n = 1$ contradicting $1 \leq m < n$.

Suppose that $k < l$ and $m = n$. If ${(-1)}^{m + 1}vba^{m} + a^{m}bv \, \not \equiv \, 0$ Corollary~\ref{ewdw} yields $d(w_{2}) = m$. Thus $k = m$
and $l = k + 1$. By \eqref{k+l} and \eqref{k+l-1} then we have
$\displaystyle \binom{2k + 1}{k} \, {\lambda}(ab^{2d + 2}) \, b \, = \, \mp \, \binom{2k}{k} \, {\lambda}(bvb)$.
Since $ab^{2d + 2}$ is a Lyndon word we have ${\gamma}(ab^{2d + 2}) = 1$, so we get $(k + 1) {\gamma}(bvb) = (2k + 1)$ which yields
$k + 1 \, | \, 2k + 1$, a contradiction. If, on the other hand, ${(-1)}^{m + 1}vba^{m} + a^{m}bv \, \equiv \, 0$ we get $d(w_{2}) \geq m + 1$.
Since $d(w_{1}) = k$ we get $k \geq m + 1$. Then we get $2k \geq 2m + 2 = k + l + 1$. It follows that $k \geq l + 1$, which contradicts $k < l$.

Suppose finally that $k = l$. Since $2k - 1 = m + n$ we necessarily get $m < n$. Corollary~\ref{ewdw} then implies that $d(w_{1}) = k$ since
${(-1)}^{k+1}b^{2d + 2}a^{k} + a^{k}b^{2d + 2} \, \not \equiv \, 0$ because the latter is equivalent to $2 \, a^{k}b^{2d + 2} \, \not \equiv \, 0$.
If $vba^{n} \, \not \equiv \, 0$ then $d(w_{2}) = m$. It follows that $m = k$ and hence $n = k - 1$, which contradicts the fact that $m < n$.
On the other hand, if $vba^{n} \, \equiv \, 0$ there exists a word $s \in A^{*}$ such that $vba^{n} = a^{n}bs{\tilde{s}}ba^{n}$.
Since ${|v|}_{a} = 1$ we must have ${|s|}_{a} = 0$ and $n = 1$. But the latter contradicts our assumption $1 \leq m < n$.

\smallskip

$\mathbf{C. \, bub \, \equiv \, 0 \, \equiv \, bvb}$. If $u = x{\tilde{x}}$ and $v = y{\tilde{y}}$ for some $x, y \in A^{*}$ then we
necessarily have $k < l$ and $m < n$. Since $x{\tilde{x}}ba^{l} \, \not \equiv \, 0$ and $y{\tilde{y}}ba^{n} \, \not \equiv \, 0$ Corollary~\ref{ewdw}
implies that $d(w_{1}) = k$, $d(w_{2}) = m$ and $e(w_{1}) = k + l - 1$, $e(w_{2}) = m + n - 1$. Thus $k = m$ and $l = n$.
Thus we are dealing with the equation $a^{k}bx{\tilde{x}}ba^{l} \, \equiv \, \pm \, a^{k}by{\tilde{y}}ba^{l}$ which has already been
considered in \S 5.

Finally, if $bub = bvb = b^{2d + 3}$, for some $d \geq 0$ then $k + l = m + n$. By \eqref{k+l-1} we also get
$\displaystyle \binom{k + l}{k} \, ab^{2d + 2} \, \equiv \, \pm \, \binom{m + n}{m} \, ab^{2d + 2}$. The minus case is immediately dispatched.
For the plus case we either get $m = k$ or $m = l$ which respectively implies either that $w_{1} = w_{2}$ or $w_{1} = \widetilde{w_{2}}$, as
required.

\bigskip
\bigskip

%%%%%%%%%%%%%%%%%%%%%%%%%%%%%%%%%%%% REFERENCES %%%%%%%%%%%%%%%%%%%%%%%%%%%%%%%%%%%%%%%%%%%%%%%%%%%%%%%%%%%%%%%%%%%%%%%%%%%%%%%%%%%%%%%%%%%%%%%%%%%%%%%% %%%%%%%%%%%%%%%%%%%%%%%%%%%%%%%%%%%%%%%%%%%%%%%%%%%%%%%%%%%%%%%%%%%%%%%%%%%%%%%%%%%%%%%%%%%%%%%%%%%%%%%%%%%%%%%%%%%%%%%%%%%%%%%%%%%%%%%%%%%%%%%%%%%%%%%%

\end{document}